\documentclass{imanum}
\usepackage{graphicx}

\usepackage{bm}
\usepackage{stmaryrd}
\usepackage{todonotes}

\newcommand*{\llbrace}{\{\mskip-6mu\{}
\newcommand*{\rrbrace}{\}\mskip-6mu\}}

\newcommand{\wt}[1]{\widetilde{#1}}
\newcommand{\mr}[1]{\mathrm{#1}}
\newcommand{\mc}[1]{\mathcal{#1}}
\newcommand{\mb}[1]{\mathbb{#1}}

\jno{drnxxx}
\received{}
\revised{}

\begin{document}

\title{Analysis of a mixed discontinuous Galerkin method for the
time-harmonic Maxwell equations with minimal smoothness
requirements}
\shorttitle{Analysis of a mixed DG method for Maxwell equations}

\author{%
{\sc
Kaifang Liu
\thanks{Corresponding author. Email: k.liu@utwente.nl}
}\\[2pt]
Department of Applied Mathematics, University of Twente,
P.O. Box 217,\\ 7500 AE Enschede, The Netherlands \\
[6pt]
{\sc Dietmar Gallistl
} \\[2pt]
Institut f\"{u}r Mathematik, Universit\"{a}t 
Jena, 07743 Jena, Germany\\ 
[6pt]
{\sc Matthias Schlottbom
and J.J.W. van der Vegt
} \\ [2pt]
Department of Applied Mathematics, University of Twente,
P.O. Box 217,\\ 7500 AE Enschede, The Netherlands
}
\shortauthorlist{K. Liu \emph{et al.}}

\maketitle

\begin{abstract}
{An error analysis of a mixed discontinuous Galerkin (DG) method
with Brezzi numerical flux for the time-harmonic Maxwell
equations with minimal smoothness requirements is presented. The
key difficulty in the error analysis for the DG method is that
the tangential or normal trace of the exact solution is not
well-defined on the mesh faces of the computational mesh. We
overcome this difficulty by two steps. First, we employ a
lifting operator to replace the integrals of the
tangential/normal traces on mesh faces by volume integrals.
Second, optimal convergence rates are proven by using smoothed
interpolations that are well-defined for merely integrable
functions. As a byproduct of our analysis, an explicit and
easily computable stabilization parameter is given.}
{discontinuous Galerkin method; time-harmonic Maxwell
equations; minimal regularity; lifting operator;
Brezzi numerical flux.}
\end{abstract}

\section{Introduction}
\label{sec:intro}
We consider the analysis of mixed discontinuous Galerkin
approximations for the time-harmonic Maxwell equations with low
regularity solutions: find $\bm{u},p$ such that 
\begin{subequations}
	\label{eq:Max}
\begin{align}
	\nabla \times(\mu^{-1}\nabla \times \bm{u})- k^2\varepsilon\bm{u}
		-\varepsilon \nabla p &= \bm{j}\quad \mbox{in}\ \Omega, 
		\label{eq:Max-1}\\
	\nabla \cdot (\varepsilon\bm{u}) &= 0\quad \mbox{in}\ \Omega,
	\label{eq:Max-2}\\
	\bm{n}\times \bm{u} &= \bm{0}\quad \mbox{on } \Gamma, 
	\label{eq:Max-3}\\
	p &= 0 \quad  \mbox{on}\  \Gamma. \label{eq:Max-4}
\end{align}
\end{subequations}
Here, $\bm{u}$ represents the electrical field, $p$ the Lagrange
multiplier used to enforce the divergence constraint
\eqref{eq:Max-2}, $k$ is the wave number and $\bm{j}\in
L^{2}(\Omega)^{3}$ is the source term. The piecewise constant
coefficients $\mu$ and $\varepsilon$ are the magnetic
permeability and electrical permittivity of the media,
respectively. We assume that $\Omega \subset
\mathbb{R}^{d},\ d=2,\ 3$ is a simply connected Lipschitz domain
with connected boundary $\Gamma$ and $\bm{n}$ is the external
unit normal vector. 

Several numerical methods for the approximation of the
time-harmonic Maxwell equations have been investigated, such as
finite-difference time-domain (FDTD) methods \citep[see,
e.g.,][]{Taflove2005, Gedney2011}, conforming finite element
methods \citep[see, e.g.,][]{Brenner2008,Ern2004,Monk2003} and
discontinuous Galerkin methods \citep[see, e.g.,][]{Ern2012,
Arnold2002, Perugia2003,Houston2005}. Standard FDTD methods suffer from serious accuracy loss near curved boundaries and
singularities, see, e.g., \cite{Nicolaides04} for a modified scheme for complex geometries and further references.
Furthermore, the corresponding error analysis for
low regularity problems is challenging; let us refer to \cite{Jovanovic2014} for the analysis of finite difference schemes for certain linear elliptic, parabolic and hyperbolic equations with minimal regularity assumptions on the solution. While curl-conforming finite element methods
(FEMs) have a proper mathematical foundation, it is difficult to
construct curl-conforming approximations in the context of
\textit{hp}-adaptation \citep[see, e.g.,][]{Monk2003,
Demkowicz2003}. In comparison, discontinuous Galerkin (DG)
methods are well suited for complex geometries,
\textit{hp}-adaptation and parallel computing.

There are several papers devoted to solving the time-harmonic
Maxwell equations using DG methods. In \cite{Perugia2002}, an
interior-penalty DG method was proposed for the indefinite
time-harmonic Maxwell equations with smooth coefficients. The
method of \cite{Perugia2002}, however, involves many terms and
parameters, which makes the practical implementation difficult.
\citet{Houston2004} introduced a mixed DG method for
\eqref{eq:Max} with $k=0$, which gives a significantly
simplified DG formulation with less terms and allows piecewise
constant coefficients $\mu$ and $\varepsilon$. Also, by adding an
auxiliary variable to transform the DG discretization into a
standard mixed formulation, i.e., a saddle-point problem without
penalty, the error analysis was greatly simplified.
Subsequently, the formulation was further simplified by removing
the standard penalization term of the normal jump at mesh
interfaces \cite[]{Houston2005b}. This simplification allowed
the use of discrete Helmholtz decomposition for the analysis of
the DG method.

In the mentioned DG methods \cite[]{Perugia2002, Houston2004,
Houston2005}, the \textit{a priori} error estimates requires
relatively high regularity of the exact solution of the
time-harmonic Maxwell equations. However, strong smoothness
assumptions are not realistic in general, since the solution of
the Maxwell equations may exhibit singularities and is non-smooth
at sharp corners and material interfaces \cite[]{Costabel1999}.
An explicit low regularity bound of the Maxwell equations can be
found, e.g., in \cite[Theorem 5.1]{Bonito2013a}. 

There are several papers devoted to the analysis of finite
element methods for the time-harmonic Maxwell equations with low
regularity solution. \cite{Ciarlet2016} proposed an error
estimate for low-regularity electromagnetic fields, where the
fields are decomposed into a regular part and a gradient, which
are approximated by the classical N\'{e}d\'{e}lec interpolation
\cite[Section 5.5]{Monk2003} and the
Cl\'{e}ment/Scott-Zhang interpolation \cite[see,
e.g.,][]{Ern2004, Brenner2008}, respectively. \cite{Ern2017b}
presented optimal error estimates for a conforming FEM for low
regularity Maxwell equations, which crucially employs recent
results on the commuting quasi-interpolation \cite[]{Ern2016}
defined on function spaces with low regularity index and their
corresponding quasi-best approximation \cite[]{Ern2017}.

The key difficulty in the error analysis of non-conforming FEMs
for non-smooth problems is that the classical trace theorems are
not applicable, i.e., the exact solution does not have  a
sufficiently regular trace on mesh faces. Until now, only a few
techniques have been developed to overcome this difficulty.  
One technique for the Maxwell equations relies on the definition
of generalized traces \cite[Proposition 7.3 and Assumption
4]{Buffa2006}. In the spirit of \cite{Buffa2006},
\cite{Bonito2016} proposed an interior-penalty method with $C^0$
finite elements for the Maxwell equations with minimal
smoothness requirements. Recently, \cite{Ern2019} analyzed a
non-conforming approximation of elliptic PDEs with minimal
regularity by introducing a generalized normal derivative of the
exact solution at the mesh faces. They also showed that this
idea can be extended to solve the time-harmonic Maxwell
equations with low regularity solutions by introducing a more
general concept for the tangential trace. Another technique that
avoids the definition of generalized traces, which has been
proposed by \cite{Gudi2010} in the context of elliptic PDEs, is
to use an enriching map to transform a non-conforming function
into a conforming one.

In this paper, we analyse a mixed DG formulation for the Maxwell
equations with low regularity solutions, which modifies the
method of \cite{Houston2004} by employing Brezzi numerical
fluxes \cite[]{Brezzi2000}. The main objective is to generalize
the error analysis of \cite{Houston2004} to the non-smooth case
and present optimal a priori error estimates for the low
regularity solution in the broken Sobolev space
$H^s(\mathcal{T}_h),s\geq 0$ with $\mathcal{T}_h$ the finite
element partition. The proof of our a priori error analysis is
different from \cite[]{Buffa2006, Bonito2013a, Ern2019} in that,
first, it employs a lifting operator that allows us to replace
integrals over faces by integrals over volumes and, thus, avoids
the definition of a generalized tangential trace on mesh faces.
Second, smoothed interpolations, which are well-defined for
merely integrable functions, are used to prove optimal
convergence rates. A further major benefit of using the lifting
operator is that we obtain an explicit expression for
stabilization parameters, which, compared to \cite{Houston2004},
facilitates the implementation considerably.

The paper is organized as follows. We introduce notation and
the variational formulation of the time-harmonic Maxwell
equations in Section \ref{sec:preliminaries}. The finite element
spaces and the mixed discontinuous Galerkin method with the
Brezzi numerical flux are presented in Section \ref{sec:FEM
spaces}. We state the main results of the paper in Section
\ref{sec:main results}. An auxiliary variational formulation in
the spirit of \cite{Houston2004} and some interpolation error
estimates are presented in Section \ref{sec:interpolations}.
Next, we first derive an error estimate for the Maxwell
equations \eqref{eq:Max-1} with $k=0$ in Section
\ref{sec:definiteMax}, and subsequently, we show in Section
\ref{sec:well-posedness} the well-posedness and error estimates
of the mixed DG method for the indefinite Maxwell equations,
i.e., $k\neq 0$. Some auxiliary results are proven in the
Appendix.

\section{Preliminaries}
\label{sec:preliminaries}

\subsection{Function spaces}

We introduce standard notation for Sobolev spaces. For a generic
bounded Lipschitz domain $D\subset \mathbb{R}^d,\ d=2,\ 3$, we
denote by $H^{m}(D)$ the usual Sobolev space of integer order
$m\geq 0$ with norms $\|\cdot \|_{m,D}$, and write $L^{2}(D) =
H^{0}(D)$. We also write $\|\cdot\|_{m,D}$ for the norm of
vector-valued function spaces $H^{m}(D)^d$. The fractional
Sobolev spaces $H^{s}(D)$ (resp. $H_0^{s}(D))$, $s\in (m,m+1)$
are defined by real interpolation between $H^m(D)$ and
$H^{m+1}(D)$, resp., $H^m_0(D)$ and $H^{m+1}_0(D)$, see, e.g.,
\cite{Tartar2007, Brenner2008}. The space $H^{s}_0(D)=[L^2(D),
H^1_0(D)]_{s,2}$ with zero trace is equivalent to the completion
of $C^{\infty}_0(D)$ with respect to the norm $\|\cdot
\|_{s,D}$, except for $s=\frac{1}{2}$. For $s=\frac{1}{2}$, it
holds that $[L^2(D), H^1_0(D)]_{1/2, 2} = H^{1/2}_{00}(D)$, but
the completion of $C^{\infty}_0(D)$ is $H^{1/2}(D)$
\cite[see][Chapter 33]{Tartar2007}. Hence, the space
$H^{1/2}_0(D)$ here is actually $H^{1/2}_{00}(D)$. We denote by
$(\cdot,\cdot)_D$ the standard inner product in $L^{2}(D)^{d}$
and denote by $L^{2}_{\varepsilon}(D)^{d}$ the space
$L^{2}(D)^{d}$ endowed with the $\varepsilon$-weighted inner
product given by $(\bm{u},\bm{v})_{\varepsilon,D} =\int_{D}
\varepsilon \bm{u} \cdot \bm{v}\,\mathrm{d}x $ for
$\varepsilon(\bm{x})$ being uniformly symmetric and positive
definite. If $D=\Omega$, we write $(\cdot,\cdot)_{\varepsilon}$
for $(\cdot,\cdot)_{\varepsilon,\Omega}$.

We will also use the following spaces
\begin{align*}
	H(\mathrm{curl},\Omega) &=\{ \bm{v}\in L^{2}(\Omega)^{d}:\
		\nabla\times \bm{v}\in L^{2}(\Omega)^{2d-3}\}, \\
	H_{0}(\mathrm{curl},\Omega) &=\{ \bm{v}\in H(\mathrm{curl},\Omega):\ 
		\bm{n}\times \bm{v}=0 \mbox{ on } \partial \Omega\},\\
	H(\mathrm{div}_{\varepsilon},\Omega) &= \{ \bm{v}\in 
		L^{2}(\Omega)^{d}:\ \nabla\cdot (\varepsilon \bm{v})\in 
		L^{2}(\Omega) \},\\
	H(\mathrm{div}_{\varepsilon}^{0},\Omega) &= \{ \bm{v}\in 
		H(\mathrm{div}_{\varepsilon},\Omega):\ \nabla\cdot 
		(\varepsilon \bm{v})=0 \}.
\end{align*}
Here, $\nabla\times \bm{u}=\left( \frac{\partial u_3}{\partial
x_2}-\frac{\partial u_2}{\partial x_3}, \frac{\partial
u_1}{\partial x_3}-\frac{\partial u_3}{\partial x_1},
\frac{\partial u_2}{\partial x_1}-\frac{\partial u_1}{\partial
x_2} \right)^T $ for $d=3$ and $\nabla\times \bm{u}
=\frac{\partial u_2}{\partial x_1}-\frac{\partial u_1}{\partial
x_2} $ for $d=2$.

The space $H_{0}(\mathrm{curl},\Omega)$ allows an $(\cdot,
\cdot)_{\varepsilon}$-orthogonal Helmholtz decomposition
\cite[Lemma 4.5]{Monk2003}
\begin{equation}
	\label{eq:Helmholtz}
	H_{0}(\mathrm{curl},\Omega)=W\oplus \nabla 
	H^1_0(\Omega),
\end{equation}
where $W=H_{0}(\mathrm{curl},\Omega)\cap
H(\mathrm{div}_{\varepsilon}^{0},\Omega)$, with compact
imbedding $W \subset \subset  L_{\varepsilon}^2(\Omega)^d$, see
\cite[Theorem 4.7]{Monk2003} for details.

\subsection{Variational formulation}

Throughout the paper we assume that the coefficients
$\mu,\varepsilon$ are piecewise constant matrix-valued functions
such that there exist positive constants $\mu_*,\mu^*,
\varepsilon_*, \varepsilon^*$ with
\begin{align}\label{eq:coefficients}
	\mu_{*}|\bm{\xi}|^2  \leq \sum_{i,j=1}^d
		\mu_{ij}(\bm{x})\xi_i\xi_j\leq \mu^{*}|\bm{\xi}|^2
		\ ~ \mbox{and }~
	\varepsilon_{*}|\bm{\xi}|^2  \leq \sum_{i,j=1}^d
		\varepsilon_{ij}(\bm{x})\xi_i\xi_j\leq \varepsilon^{*}|\bm{\xi}|^2,
\end{align}
for a.e. $\bm{x}\in \overline{\Omega}$, and all vectors
$\bm{\xi}\in \mathbb{R}^d$. More precisely, we assume that $\mu
$ and $\varepsilon$ are piecewise constant with respect to some
partition $\mathcal{T}_h$ of $\Omega$ into Lipschitz polyhedra.
In the following, we also assume that $k^{2}$ is not an interior
Maxwell eigenvalue, see \cite[Section 1.4.2]{Monk2003} or
\cite[(11.2.6)]{Boffi2013} for a definition. 

Let $V:=H_{0}(\mathrm{curl},\Omega) $ and
$Q:=H^{1}_{0}(\Omega)$. Define the bilinear forms $a(\cdot
,\cdot )$ and $b(\cdot ,\cdot )$ as
\begin{alignat*}{3}
	a(\bm{u},\bm{v})&=(\mu^{-1}\nabla\times \bm{u}, \nabla\times \bm{v}), \qquad 
		&& \forall\, \bm{u},\bm{v}\in V,  \\
	b(\bm{v},p)&=-(\varepsilon \bm{v},\nabla p), && \forall\, \bm{v}\in 
		V,\, p \in Q.
\end{alignat*}
The mixed variational formulation of the time-harmonic Maxwell
equations is to find $\bm{u}\in V$ and $p \in Q$ such that 
\begin{alignat}{3}
	a(\bm{u},\bm{v})-k^2(\varepsilon\bm{u},\bm{v})+b(\bm{v},p) &= 
	(\bm{j},\bm{v}), \qquad && \forall\, \bm{v}\in V,  \label{eq:conForm1}\\
	b(\bm{u},q)&=0,  && \forall\, q\in Q. \label{eq:conForm2}
\end{alignat}

Due to that $a(\cdot,\cdot)$ is continuous and coercive on the
kernel of $b$, and $b(\cdot,\cdot)$ is continuous and satisfies
the inf-sup condition, see \cite[Section 2.3]{Houston2005a} or
\cite[Theorem 11.2.1]{Boffi2013}, the variational problem is
well-posed.

\begin{lemma}[Theorem 11.2.1, \cite{Boffi2013}]
	\label{lem:uniqueness}
	Suppose that $k^2$ is not a Maxwell eigenvalue. The
	variational problem \eqref{eq:conForm1}-\eqref{eq:conForm2}
	has a unique solution $(\bm{u},p)\in V\times Q$, and there
	exists a constant $C>0$ such that 
	\[
		\|\bm{u}\|_{H(\mathrm{curl},\Omega)}+\|p\|_{1,\Omega} 
		\leq C \|\bm{j}\|_{0,\Omega}.
	\]
\end{lemma}

The following stability results, which were proven in
\cite{Bonito2013a}, are very useful for our error estimates.

\begin{lemma}[Theorem 5.1, \cite{Bonito2013a}]
	\label{lem:regularity}
	Suppose that \eqref{eq:coefficients} holds. Then the weak
	solution $(\bm{u},p)\in V\times Q$ of the variational
	problem \eqref{eq:conForm1}-\eqref{eq:conForm2} satisfies 
	\begin{align*}
		\|\bm{u}\|_{s,\Omega} & \leq C 
		\|\bm{j}\|_{0,\Omega}, \quad \forall\, 
		0\leq s <\tau_{\varepsilon},\\
		\|\nabla\times \bm{u}\|_{s,\Omega} & \leq C 
		\|\bm{j}\|_{0,\Omega}, \quad \forall\, 
		0\leq s <\tau_{\mu},\\
		\|\nabla\times (\mu^{-1}\nabla\times \bm{u})\|_{0, \Omega}
		+ \|\nabla p\|_{0,\Omega} & \leq C 
		\|\bm{j}\|_{0,\Omega},
	\end{align*}
	where the positive constants $C$, $\tau_{\varepsilon}$ and
	$\tau_{\mu}$ depend only on $\Omega$ and $\varepsilon$ and
	$\mu$.
\end{lemma}

Note that in general the differentiability indices
$\tau_{\varepsilon},\tau_{\mu}$ are less than $1/2$ for
Lipschitz domains and discontinuous coefficients $\varepsilon$
and $\mu$ \cite[see][]{Bonito2016}.

\section{Mixed discontinuous Galerkin discretization}
\label{sec:FEM spaces}

\subsection{Finite element spaces}

Let $\mathcal{T}_{h}$ be a shape regular partition of the domain
$\Omega$ into tetrahedra, such that the coefficients are
constant on each $K\in \mathcal{T}_{h}$. We denote by $h_{K}$
the diameter of an element $K$ and denote $h=\max_{K\in
\mc{T}_h}h_K$. For an integer $\ell\geq 0$ and an element $K \in
\mathcal{T}_{h}$, we define $\mathbb{P}_{\ell}(K)$ as the space
of polynomials of total degree $\ell$ in $K$. Let
$\mathcal{F}_{h}$ be the union of interior faces
$\mathcal{F}_{h}^{I}$ and boundary faces $\mathcal{F}_{h}^{B}$.
For piecewise smooth vector- or scalar-valued functions $\bm{v}$
and $q$, we define jumps and averages at faces in the mesh
$\mathcal{T}_{h}$. Let $F \in \mathcal{F}_{h}^{I}$ be an
interior face shared by two elements $K^{+}$ and $K^{-}$ and let
$\bm{n}^\pm$ be the unit outward normal vectors on the
boundaries $\partial K^{\pm}$, then the tangential and normal
jumps across $F$ are, respectively, defined by $\llbracket
\bm{v} \rrbracket_{T}=\bm{n}^{+}\times \bm{v}^{+} +
\bm{n}^{-}\times \bm{v}^{-}$, $\llbracket \bm{v}
\rrbracket_{N}=\bm{v}^{+}\cdot \bm{n}^{+} + \bm{v}^{-}\cdot
\bm{n}^{-}$, and $\llbracket q \rrbracket_{N} =
q^{+}\bm{n}^{+}+q^{-}\bm{n}^{-}$. We also define the averages by
$\llbrace \bm{v}\rrbrace=(\bm{v}^{+}+ \bm{v}^{-})/2$, and
$\llbrace q\rrbrace=(q^{+}+ q^{-})/2$. If $F \in
\mathcal{F}_{h}^{B}$ is a boundary face, we set $\llbracket
\bm{v} \rrbracket_{T}=\bm{n}\times \bm{v}$, $\llbracket \bm{v}
\rrbracket_{N}=\bm{v}\cdot\bm{n} $, $\llbracket q
\rrbracket_{N}=q\bm{n}$, $\llbrace \bm{v}\rrbrace=\bm{v}$ and
$\llbrace q\rrbrace=q$.

We define the broken Sobolev spaces with respect to the partition
$\mathcal{T}_h$ of $\Omega$ as
\[
	H^{s}(\mathcal{T}_{h})^{d}=\{ \bm{v}\in L^{2}(\Omega)^{d}:\ 
	\bm{v}_{|K}\in H^{s}(K)^{d}, \quad 
	\forall\, K \in \mathcal{T}_{h}\},
\]
with norm 
\[
	\|\bm{v}\|_{s,\mathcal{T}_{h}}^{2}=\sum_{K \in \mathcal{T}_{h}}
		\|\bm{v}\|_{s,K}^{2}.
\]
Moreover, we define the finite element spaces without
inter-element continuity condition as 
\begin{align}
	V_{h} &:= \{ \bm{u} \in L^{2}(\Omega)^{d}:\ 
		\bm{u}_{|K}\in R_{\ell}(K),\quad  \forall\, K \in 
		\mathcal{T}_{h}\}, \label{eq:Vh} \\
	Q_{h} &:= \{ q \in L^{2}(\Omega):\ q_{|K}\in 
		\mathbb{P}_{\ell}(K),\quad  \forall\, K \in \mathcal{T}_{h}\},
		\label{eq:Qh}
\end{align} 
where $R_{\ell}$ denotes the space of N{\'e}d{\'e}lec functions
of degree $\ell$, i.e., for $d=3$,
$R_{\ell}=(\mathbb{P}_{\ell-1})^{3}\oplus S_{\ell}$, and
$S_{\ell}=\{ \bm{q} \in (\wt{\mathbb{P}}_{\ell})^{3}:\
\bm{x}\cdot \bm{q}=0 \}$ with $\wt{\mb{P}}_{\ell}$ being the
homogeneous polynomials of degree $\ell$, see \cite[][Remark
5.29]{Monk2003} for $d=2$. 

We also define $H(\mathrm{curl})$-conforming subspaces with and
without vanishing tangential trace on the boundary,
respectively, as
\begin{equation}
	\label{eq:Vhc}
	V_{h0}^{\mathrm{c}}:=V_h\cap H_{0}(\mathrm{curl},\Omega)
	\quad \mbox{and} \quad
	V_h^{\mathrm{c}}:=V_h\cap H(\mathrm{curl},\Omega).
\end{equation}

\subsection{Lifting operator}

The following lifting operators are useful in the DG
discretization by replacing the penalty terms over faces by
volume integrals, which make sense also for low regularity
functions. First, we define the local lifting operator
$\mathcal{R}_{F}:\ L^{2}(F)^{3}\rightarrow V_{h}$ on a single
face $F \in \mathcal{F}_{h}$ by 
\begin{equation}
	\label{eq:lifting}
	\int_{\Omega}\mathcal{R}_{F}(\eta)\cdot \bm{v} \,\mathrm{d}\bm{x}
	=\int_{F}\eta \cdot \llbrace \bm{v}\rrbrace \,\mathrm{d}s,
\end{equation} 
for all $\bm{v} \in V_{h}$. Since the right hand side is nonzero
only when $\llbrace\bm{v}\rrbrace$ has support on $F$, the
support of the lifting $\mathcal{R}_F(\eta)$ is limited to the
elements adjacent to face $F$. Next, the local lifting operator
can be used to define a global one. Define
$\mathcal{R}:L^2(\mathcal{F}_h)^3 \to V_h$ as 
\[
	\mathcal{R}(\eta)=
	\sum_{F \in \mathcal{F}_{h}}\mathcal{R}_{F}(\eta),
	\quad \forall \eta\in L^2(\mathcal{F}_h)^3.
\] 
\citet[Lemma 1 and 2]{Sarmany2010}
showed the following stability property of the local lifting
operator: there exist positive constants $C_{1}$ and $C_{2}$
such that for any $\bm{v}\in V_{h}$
\begin{equation}\label{eq:lifting stability}
	C_{1}h_F^{-1/2}\|\llbracket \bm{v} \rrbracket_{T}\|_{0,F} 
	\leq \|\mathcal{R}_F(\llbracket \bm{v} \rrbracket_{T})\|_{0,\Omega} 
	\leq C_{2}h_F^{-1/2}\|\llbracket 
	\bm{v} \rrbracket_{T}\|_{0,F},\quad \forall F
	\in \mathcal{F}_h,
\end{equation}
where $h_{F}$ denotes the diameter of face $F$ and $C_1,\ C_2$
are independent of $\bm{v}$.

\subsection{Mixed DG discretization}

The mixed discontinuous Galerkin discretization with Brezzi
numerical flux \cite[see, e.g.,][]{Brezzi2000} for the
time-harmonic Maxwell equations is: Find $(\bm{u}_{h},p_{h})\in
V_{h}\times Q_{h}$ such that 
\begin{alignat}{3}
	a_{h}(\bm{u}_{h},\bm{v})-k^2(\varepsilon\bm{u}_{h},\bm{v})
		+b_h(\bm{v},p_{h}) &= 	(\bm{j},\bm{v}), 
		\qquad && \forall \bm{v}\in V_h, 	\label{eq:MDG1} 	\\
	b_{h}(\bm{u}_{h},q)-c_{h}(p_{h},q) &=0,  
		&& \forall q\in Q_h,  	\label{eq:MDG2}
\end{alignat}
where,
\begin{align*}
	a_{h}(\bm{u},\bm{v}) 
	&= 	\left(\mu^{-1} \nabla_{h}\times \bm{u}, 
			\nabla_{h}\times \bm{v} \right)
			- \left( \mathcal{R}(\llbracket \bm{u} \rrbracket_{T}),
			\mu^{-1}\nabla_h\times \bm{v} \right) \\
		&	- \left( \mathcal{R}(\llbracket \bm{v} \rrbracket_{T}),
			\mu^{-1}\nabla_h\times \bm{u} \right)
			+ \sum_{F \in \mathcal{F}_{h}} \alpha_F
			\left(\mu^{-1} \mathcal{R}_{F}(\llbracket \bm{u} \rrbracket_{T}),
			\mathcal{R}_{F}(\llbracket \bm{v} \rrbracket_{T}) \right), \\
	b_{h}(\bm{v},p) 
	&=  -\left( \varepsilon\bm{v}, \nabla_{h}p \right) 
		+ \left( \mathcal{R}(\llbracket p \rrbracket_{N}),
		\varepsilon \bm{v} \right), \\
	c_{h}(p,q)
	&= 	\sum_{F \in \mathcal{F}_{h}}\gamma_{F}
		\left( \varepsilon\mathcal{R}_{F}(\llbracket p\rrbracket_{N}),
		\mathcal{R}_{F}(\llbracket q\rrbracket_{N}) \right).
\end{align*}
Here, $\nabla_{h}$ and $\nabla_h\times$ denote the elementwise
action of the differential operators $\nabla$ and
$\nabla\times$, respectively. We set the parameter $\gamma_{F}$
strictly positive for all $F \in \mathcal{F}_{h}$, and
$\alpha_F>0$ will be chosen later. The readers are referred to
\cite{Houston2004} for the derivation of the DG formulation with
the Brezzi numerical flux replaced by the interior penalty
numerical flux.

\begin{remark}
	The main difference between the mixed DG formulation
	\eqref{eq:MDG1}-\eqref{eq:MDG2} and those discussed in
	\cite[][]{Houston2004, Lu2017} is the use of the lifting
	operator $\mathcal{R}_F$. Two main benefits of using the
	lifting operator are:
	\begin{enumerate}
		\item[$(\mr{i})$] precise condition for $\alpha_F$ can
		be computed from the mesh $\mathcal{T}_{h}$ that ensures
		stability, see Proposition \ref{pop:coercivity}, while
		in practice for interior-penalty formulations,
		depending on the computational mesh, the penalty
		parameters in \cite[][]{Houston2004, Lu2017} need to be
		frequently adjusted;
		\item[$(\mr{ii})$] the bilinear forms $a_h$ and $b_h$ can be
		trivially extended to 
		\[
			a_h:V\times V\to \mathbb{R},\quad
			\mbox{and}\quad 
			b_h:V\times Q \to \mathbb{R}.
		\]
		This avoids the generalization of the tangential trace
		on element faces for low regularity solutions, which
		causes great technical difficulties for nonconforming
		finite element methods for solving problems with low
		regularity solution, see, e.g., \cite{Buffa2006,
		Bonito2016, Ern2019}.
	\end{enumerate}
\end{remark}

\begin{remark}
	\label{rmk:well-defined}
	Since the coefficients $\mu$ and $\varepsilon$ are piecewise
	constant, we note from the definition of the lifting
	operator $\mathcal{R}$ that the bilinear forms
	$a_h:V_h\times V_h\to \mathbb{R}$ and $b_h:V_h \times Q_h\to
	\mb{R}$ have the following equivalent expressions
	\begin{align*}
		a_{h}(\bm{u},\bm{v}) 
		&= 	\left(\mu^{-1}\nabla_{h}\times \bm{u}, 
			\nabla_{h}\times \bm{v} \right) 
			- \int_{\mathcal{F}_{h}} \llbracket \bm{u} \rrbracket_{T} 
			\cdot \llbrace \mu^{-1}\nabla_{h}\times 
			\bm{v}\rrbrace\,\mathrm{d}s \\
		&	-\int_{\mathcal{F}_{h}} \llbracket 
			\bm{v} \rrbracket_{T}\cdot \llbrace \mu^{-1}\nabla_{h}\times 
			\bm{u}\rrbrace\,\mathrm{d}s   
			+ \sum_{F \in \mathcal{F}_{h}} \alpha_F
			\left( \mu^{-1}\mathcal{R}_{F}(\llbracket \bm{u} \rrbracket_{T}),
			\mathcal{R}_{F}(\llbracket \bm{v} \rrbracket_{T}) \right), \\
		b_{h}(\bm{v},p) 
		&=  -\left( \varepsilon\bm{v}, \nabla_{h}p \right) 
			+ \int_{\mathcal{F}_{h}}\llbrace 
			\varepsilon \bm{v}\rrbrace\cdot \llbracket p\rrbracket_{N}\,
			\mathrm{d}s,
	\end{align*}
	for all $\bm{u},\bm{v}\in V_h$ and $p\in Q_h$. Here,
	$\int_{\mathcal{F}_{h}}\cdot \,\mathrm{d}s =\sum_{F \in
	\mathcal{F}_{h}} \int_{F}\cdot \,\mathrm{d}s$. Note that
	compared with the implementation of the DG discretization in
	\cite{Houston2004}, we see that only the penalty terms need to
	be changed.
\end{remark}

\section{The main results}
\label{sec:main results}

We now give explicit bounds on the stabilization parameter
$\alpha_F$ that ensures well-posedness of
\eqref{eq:MDG1}-\eqref{eq:MDG2}. Subsequently, we present an a
priori error estimate for low regularity solutions.

We start with defining $V(h)=V+V_{h}$ and $Q(h)=Q+Q_{h}$ endowed
with the semi-norms and norms 
\begin{align*}
	|\bm{v}|_{V(h)}^{2} 
	&= \|\mu^{-\frac{1}{2}}\nabla_{h}\times \bm{v}\|_{0,\Omega}^{2}
	+\sum_{F \in \mathcal{F}_{h}} \|\mu^{-\frac{1}{2}}
	\mathcal{R}_F(\llbracket \bm{v} \rrbracket_{T})\|_{0,\Omega}^{2},\\
	\|\bm{v}\|_{V(h)}^{2} 
	&= \|\varepsilon^{\frac{1}{2}}\bm{v}\|_{0,\Omega}^{2} 
		+|\bm{v}|_{V(h)}^{2},\\
	\|q\|_{Q(h)}^{2} 
	&= \|\varepsilon^{\frac{1}{2}}\nabla_{h}q\|_{0,\Omega}^{2} 
		+\sum_{F \in \mathcal{F}_{h}}\|\varepsilon^{\frac{1}{2}}
		\mathcal{R}_F(\llbracket q \rrbracket_{N})\|_{0,\Omega}^{2}.
\end{align*}

The following proposition shows $a_h(\cdot,\cdot)$ is coercive
on $V(h)$ respect to the semi-norm $|\cdot|_{V(h)}$ for a simple
and explicit choice of the parameter $\alpha_F$, which
facilitates the implementation of the DG method and is essential
in the proof of the well-posedness of the mixed DG
discretization.
\begin{proposition} [Coercivity]
	\label{pop:coercivity}
	The bilinear form $a_h(\cdot, \cdot)$ satisfies
	\[
		a_{h}(\bm{v},\bm{v})\geq \frac{1}{2} |\bm{v}|_{V(h)}^{2},
		\quad \forall \bm{v}\in V(h)
	\]
	provided $\alpha_{F}\geq \frac{1}{2}+2n_{K}$. Here,
	$n_{K}=d+1$ is the number of faces of an element $K\in
	\mathcal{T}_{h}$.
\end{proposition}
\begin{proof}
	From the definition of the lifting operator $\mc{R}_F$, we
	have for any $\bm{v}\in V(h)$
	\begin{align*}
		a_{h}(\bm{v},\bm{v}) 
		=& \|\mu^{-\frac{1}{2}}\nabla_{h}\times \bm{v}\|_{0,\Omega}^{2}
		-2 \sum_{F \in \mathcal{F}_{h}}(\mathcal{R}_{F}
		(\llbracket \bm{v} \rrbracket_{T}), \mu^{-1}\nabla_{h}
		\times \bm{v}) \\ 
		&+ \sum_{F \in \mathcal{F}_{h}}\alpha_{F}
		\|\mu^{-\frac{1}{2}}\mathcal{R}_{F}(\llbracket \bm{v} 
		\rrbracket_{T})\|_{0, \Omega}^{2}  .
	\end{align*}
	Recall that the support of the local lifting operator
	$\mathcal{R}_F(\llbracket \bm{v}\rrbracket_T)$, denoted by
	$\omega_F$, consists of the element(s) adjacent to face $F$.
	By using the Cauchy-Schwarz and Young's inequality, there
	holds
	\begin{align*}
		2 \sum_{F \in \mathcal{F}_{h}}
		&(\mathcal{R}_{F}(\llbracket \bm{v} \rrbracket_{T}),
			\mu^{-1}\nabla_{h}\times \bm{v})\\
		\leq & 2 \sum_{F \in \mathcal{F}_{h}} 
			\|\mu^{-\frac{1}{2}}\mathcal{R}_{F}	(\llbracket \bm{v} 
			\rrbracket_{T})\|_{0,\Omega}	
			\|\mu^{-\frac{1}{2}}\nabla_{h}\times \bm{v}\|
			_{0, \omega_F}\\
		\leq & 2 \sum_{F \in \mathcal{F}_{h}}
			\left( \frac{1}{4\delta} 
			\|\mu^{-\frac{1}{2}}\mathcal{R}_{F} (\llbracket \bm{v} 
			\rrbracket_{T})\|_{0,\Omega}^2 +	
			\delta\|\mu^{-\frac{1}{2}}\nabla_{h}\times \bm{v}\|
			_{0, \omega_F}^2 \right)\\
		\leq & \sum_{F \in \mathcal{F}_{h}}\frac{	1}{2 \delta}
			\| \mu^{-\frac{1}{2}}\mathcal{R}_{F}(\llbracket \bm{v} \rrbracket_{T})\|_{0,\Omega}^{2}
			+2 \delta n_{K} 
			\|\mu^{-\frac{1}{2}}\nabla_{h}\times \bm{v}\|_{0,\Omega}^{2},
	\end{align*}
	where $n_{K}$ is the number of faces of an element $K\in
	\mathcal{T}_{h}$. Hence,
	\begin{align*}
		a_{h}(\bm{v},\bm{v}) 
		&\geq  \left( 1-2 \delta n_{K} \right)\|\mu^{-\frac{1}{2}}
			\nabla_{h}\times \bm{v}\|_{0,\Omega}^{2}
			+ \sum_{F \in \mathcal{F}_{h}}\left(\alpha_{F}-
			\frac{1}{2 \delta} \right)\|\mu^{-\frac{1}{2}}\mathcal{R}_{F}
			(\llbracket \bm{v} \rrbracket_{T})\|_{0,\Omega}^{2}  \\
		&\geq \min\{ 1-2 \delta n_{K}, \alpha_{F}-
		\frac{1}{2 \delta} \}|\bm{v}|_{V(h)}^{2}.
	\end{align*}
	Hence, by setting $\delta=\frac{1}{4n_{K}}$, we can get the
	coercivity with constant $\frac{1}{2}$ if $\alpha_{F} \geq
	\frac{1}{2}+2n_{K}$.
\end{proof}

The following two theorems state the well-posedness of the mixed
DG method \eqref{eq:MDG1}-\eqref{eq:MDG2} and a priori error
estimates.

\begin{theorem} [Existence, uniqueness]
	\label{thm:indefUniqueness}
	Suppose $k^2$ is not a Maxwell eigenvalue, $\alpha_F\geq
	\frac{1}{2}+2n_K$ and $\gamma_F \geq \frac{1}{2}$. For all
	mesh size $h$ small enough, there exists a unique solution
	$(\bm{u}_h,p_h)\in V_h\times Q_h$ to
	\eqref{eq:MDG1}-\eqref{eq:MDG2} satisfying the estimate
	\[
		\|\bm{u}_h\|_{V(h)}+\|p_h\|_{Q(h)}\leq C \|\bm{j}\|_{0},
	\]
	where the constant $C>0$ is independent of the mesh size and
	the solution $(\bm{u}_h,p_h)$.
\end{theorem}

\begin{theorem} [A priori error estimate]
\label{thm:defEnergyNorm}
	Suppose $k^2$ is not a Maxwell eigenvalue, $\alpha_F\geq
	\frac{1}{2}+2n_K$ and $\gamma_F \geq \frac{1}{2}$. Let
	$(\bm{u},p)$ be the solution of
	\eqref{eq:conForm1}-\eqref{eq:conForm2} with regularity
	\[
		\varepsilon\bm{u},\ \mu^{-1}\nabla\times \bm{u} 
		\in H^s(\mc{T}_h),\ p\in H^{1+s}(\mc{T}_h), \ s\geq 0.
	\]
	Then, for all mesh sizes $h$ small enough, there exists a
	unique solution $(\bm{u}_{h}, p_{h}) \in V_h\times Q_h$ to
	\eqref{eq:MDG1}-\eqref{eq:MDG2} and it satisfies the a
	priori error estimates
	\begin{align*}
		\|\bm{u}-\bm{u}_{h}\|_{V(h)} 
		\leq  & C  \Bigg( \sum_{K\in \mathcal{T}_h} \Big(  
			h_K^{2\min \{s,\ell\}}\|\varepsilon \bm{u}\|_{s,D_K}^2
			+h_K^{2\min \{s,\ell\}}\|\mu^{-1}\nabla\times \bm{u}\|_{s,D_K}^2 \\
			&+ \chi(s)h_K^{2}\|\nabla\times \mu^{-1}\nabla\times \bm{u} \|_{0,D_K}^2
			+ h_K^{2\min \{s,\ell\}}\|p\|_{1+s,D_K}^2 \Big) \Bigg)^{1/2},\\
		\|p-p_{h}\|_{Q(h)}
		\leq & C  \Bigg( \sum_{K\in \mathcal{T}_h} \Big(  
			h_K^{2\min \{s,\ell\}}\|\varepsilon \bm{u}\|_{s,D_K}^2
			+h_K^{2\min \{s,\ell\}}\|\mu^{-1}\nabla\times \bm{u}\|_{s,D_K}^2 \\
		&+ \chi(s)h_K^{2}\|\nabla\times \mu^{-1}\nabla\times \bm{u} \|_{0,D_K}^2
		+ h_K^{2\min \{s,\ell\}}\|p\|_{1+s,D_K}^2 \Big) \Bigg)^{1/2},
	\end{align*}
	where the constant $C>0$ depends on the bounds
	\eqref{eq:coefficients}, wave number $k$ and polynomial
	degree $\ell$, but is independent of the mesh size. Here,
	$\chi(s)=1$ if $s\leq \frac{1}{2}$ and $\chi(s)=0$
	otherwise, and $D_K := \mathrm{int}
	\left(\cup_{\bar{K}^{\prime}\cap \bar{K}\neq
	\emptyset}\bar{K}^{\prime} \right)$. Note that in the
	estimates $D_K$ can be replaced by $K$ for all
	$s>\frac{1}{2}$.
\end{theorem}

\begin{remark}
	We note that all above results also hold true with the
	choice of N\'{e}d\'{e}lec elements of the second type for
	$V_h$ and a full polynomial space of order $\ell+1$ for
	$Q_h$, see \cite[Section 7.1]{Buffa2006} and \cite[Section
	8.2]{Monk2003} for more details.
\end{remark}

Theorem~\ref{thm:indefUniqueness} will be proved for $k=0$ in
the next section using an auxiliary formulation in the spirit of
\cite{Houston2004}. The case $k\neq 0$ is treated in
Section~\ref{sec:well-posedness}.
Theorem~\ref{thm:defEnergyNorm} will be proven in Section
\ref{sec:definiteMax} for $k=0$ and in Section
\ref{sec:well-posedness} for $k\neq 0$, respectively.


\section{Auxiliary results}
\label{sec:interpolations}


\subsection{Auxiliary mixed formulation}
\label{subs:auxiliary}


The variational problem \eqref{eq:MDG1}-\eqref{eq:MDG2} is a saddle-point problem with penalty, to facilitate its analysis we follow \cite[Section 4 and 5]{Houston2004} and introduce an equivalent auxiliary mixed formulation, that is a saddle-point problem without penalty.
%
To do so, let us introduce the discrete auxiliary space
\[
	M_{h}:= 
	\{ \lambda \in L^{2}(\mathcal{F}_{h})^{d}:\ \lambda_{|F} \in 
	\mb{P}_{\ell}(F)^d, \quad \forall\, F \in \mathcal{F}_{h}\},
\]
with norm 
\[
	\|\lambda\|_{M_{h}}^2:=\sum_{F \in \mathcal{F}_{h}} \|\varepsilon^{\frac{1}{2}}
	\mathcal{R}_F (\lambda)\|_{0,\Omega}^2,
\]
and let $W_h=V_h\times M_{h}$ and $W(h)=V(h)\times
M_{h}$ with semi-norm and norm defined as:
\[
	|(\bm{v},\eta)|_{W(h)}^{2}=|\bm{v}|_{V(h)}^{2}+\|\eta\|_{
	M_{h}}^{2},	\quad \|(\bm{v},\eta)\|_{W(h)}^{2}=\|\bm{v}\|_{
	V(h)}^{2}+\|\eta\|_{M_{h}}^{2}.
\]

We state the auxiliary mixed formulation as follows: Find
$(\bm{u}_{h},\lambda_{h}; p_{h})\in W_h \times Q_{h}$ such
that 
\begin{alignat}{3}
	A_{h}(\bm{u}_{h},\lambda_{h};\bm{v},\eta)
	-k^2( \varepsilon \bm{u}_h,\bm{v})
	+B_{h}(\bm{v},\eta;p_{h}) &=(\bm{j},\bm{v}),
		\quad && \forall\, (\bm{v},\eta)\in W_{h},
		\label{eq:auxiliary1}\\
	B_{h}(\bm{u}_{h},\lambda_{h};q) &= 0,  && \forall\, q \in Q_{h},
		\label{eq:auxiliary2}
\end{alignat}
where 
\begin{align*}
	A_{h}(\bm{u}_h,\lambda_h;\bm{v},\eta)&=a_{h}(\bm{u}_h,\bm{v})
		+\sum_{F \in \mathcal{F}_{h}} \gamma_{F}\left( \varepsilon
		\mathcal{R}_{F}(\lambda_h) , \mathcal{R}_{F}(\eta)\right),\\
	B_{h}(\bm{v},\eta;p_h) &=b_{h}(\bm{v},p_h)-\sum_{F \in \mathcal{
		F}_{h}}	\gamma_{F}  \left(\varepsilon \mathcal{R}_{F}
		(\llbracket p_h \rrbracket_{N}), \mathcal{R}_{F}(\eta) \right).
\end{align*}

\begin{lemma} \label{lem:same solution} 
	The mixed DG formulation \eqref{eq:MDG1}-\eqref{eq:MDG2} is
	equivalent to \eqref{eq:auxiliary1}-\eqref{eq:auxiliary2},
	i.e., if $(\bm{u}_h,p_h)\in V_h \times Q_{h}$ solves
	\eqref{eq:MDG1}-\eqref{eq:MDG2}, then $(\bm{u}_h, \llbracket
	p_h\rrbracket_{N}; p_h)\in W_h \times Q_{h}$ solves
	\eqref{eq:auxiliary1}-\eqref{eq:auxiliary2}. If, on the
	other hand, $(\bm{u}_h, \lambda_h;
	p_h)\in W_h \times Q_{h}$ solves \eqref{eq:auxiliary1}-\eqref{eq:auxiliary2},
	then $(\bm{u}_h,p_h)$ solves \eqref{eq:MDG1}-\eqref{eq:MDG2}
	and $\lambda_h=\llbracket p_h\rrbracket_{N}$. 
\end{lemma}
\begin{proof}
	Suppose $(\bm{u}_h, \lambda_h; p_h)$ solves
	\eqref{eq:auxiliary1}-\eqref{eq:auxiliary2}. By taking test
	function $\bm{v}=0$ in \eqref{eq:auxiliary1}, we have
	\[
		\sum_{F \in \mathcal{F}_{h}}\gamma_{F} (\varepsilon
		\mathcal{R}_{F}(\lambda_{h}) , \mathcal{R}_{F}(\eta)) = 
		\sum_{F \in \mathcal{F}_{h}}\gamma_{F}  (\varepsilon\mathcal{R}_{F} 
		(\llbracket p_{h} \rrbracket_{N}), \mathcal{R}_{F}(\eta)), 
		\quad \forall\, \eta \in M_{h}.
	\]
	Hence, $\lambda_{h}=\llbracket p_{h} \rrbracket_{N}$. This shows that $(u_h,p_h)$ solves
	\eqref{eq:MDG1}-\eqref{eq:MDG2}. The other direction follows immediately by setting $\lambda_h=\llbracket p_{h} \rrbracket_{N}$. 
\end{proof}

Define the kernel of the form $B_{h}(\cdot,\cdot)$ as 
\begin{equation}
	\label{eq:Zh}
	\mr{Ker}(B_{h}):=\{ (\bm{v},\eta)\in W_{h}:\ 
	B_{h}(\bm{v},\eta;q)=0, \ \forall\, q \in Q_{h} \}.
\end{equation}

The following three lemmas form the basis for the proof of Theorem~\ref{thm:indefUniqueness} for $k=0$.

\begin{lemma} [Continuity]
	\label{lem:continuity}
	There exists a constant $C$ independent of the mesh size and
	the coefficients $\mu$ and $\varepsilon$ such that 
	\begin{alignat*}{3}
		|A_{h}(\bm{u},\lambda;\bm{v},\eta)|	
			&\leq C\|(\bm{u},\lambda)\|_{W(h)}\|(\bm{v},\eta)
			\|_{W(h)}, \qquad
			&&\forall\, (\bm{u},\lambda),(\bm{v},\eta) \in W(h), \\
		|B_{h}(\bm{v},\eta;q)| 
			&\leq C \|(\bm{v},\eta)\|_{W(h)} 
			\|q\|_{Q(h)}, &&\forall\, (\bm{v},\eta) \in W(h),\ 
			\forall\, q \in Q(h).
	\end{alignat*}
\end{lemma}
This lemma follows directly from an application of the
Cauchy-Schwarz inequality.

\begin{lemma} [Ellipticity on the kernel]
	\label{lem:ellipticity}
	For $\alpha_F$ given by Proposition \ref{pop:coercivity} and
	$\gamma_F\geq \frac{1}{2}$, there holds
	\begin{equation} \label{eq:ellipticity}
		A_{h}(\bm{v},\eta;\bm{v},\eta)\geq \kappa_{A}\|(\bm{v},\eta)
		\|_{W(h)}^{2},\quad \forall\, (\bm{v},\eta)\in \mr{Ker}(B_{h}),
	\end{equation}
	where $\kappa_{A}>0$ depends on the coefficients $\mu$ and
	$\varepsilon$ but independent of the mesh size.
\end{lemma}

\begin{proof}
	Assume $(\bm{v},\eta)\in \mr{Ker}(B_h)$. Recalling the
	definition of $A_{h}$ and using the coercivity of $a_{h}$
	stated in Proposition \ref{pop:coercivity}, there holds 
	\begin{equation}
		\label{eq:ellipticity-01}
		A_{h}(\bm{v},\eta;\bm{v},\eta) 
		\geq \frac{1}{2} |\bm{v}|_{V{(h)}}^{2}
			+ \sum_{F \in \mathcal{F}_{h}}\gamma_{F}
			\|\mathcal{R}_{F}(\eta)\|_{0,\Omega}^{2}
		\geq \frac{1}{2}|(\bm{v},\eta)|_{W(h)}^{2}.
	\end{equation}
	From the discrete Friedrichs inequality in Appendix
	\ref{lem:discrete Friedrichs}, we have
	\[
		\|\varepsilon^{\frac{1}{2}}\bm{v}\|_{0,\Omega}
		\leq c_F 
		|\bm{v}|_{V(h)},
	\]
	which for $\delta>0$ leads to
	\begin{align*}
		|\bm{v}|_{V(h)}^2 = (1-\delta)|\bm{v}|_{V(h)}^2
		+ \delta |\bm{v}|_{V(h)}^2 
		\geq (1-\delta)|\bm{v}|_{V(h)}^2
		+ \frac{\delta}{c_F^2} 
		\|\varepsilon^{\frac{1}{2}}\bm{v}\|_{0,\Omega}^2.
	\end{align*}
	By setting $\delta=\frac{c_F^2}{1+c_F^2}$, we have 
	\[
		|\bm{v}|_{V(h)}^2 \geq \frac{1}{1+c_F^2} \|\bm{v}\|_{V(h)},
	\]
	which, together with \eqref{eq:ellipticity-01}, completes
	the proof with $\kappa_A=\frac{1}{2(1+c_F^2)}$.
\end{proof}

The following stability result follows with similar arguments as
in \cite[Proposition 5.4]{Houston2004} and using the stability
of the lifting operator \eqref{eq:lifting stability}.

\begin{lemma} [Inf-sup condition]
	\label{lem:inf-sup}
	There holds
	\begin{equation} \label{eq:inf-sup}
		\inf_{0\neq q \in Q_{h}} \sup_{0\neq (\bm{v},\eta)\in 
		W_{h}} \frac{B_{h}(\bm{v},\eta;q)}{\|q\|_{Q(h)} 
		\|(\bm{v},\eta)\|_{W(h)}}
		\geq \kappa_{B}>0,
	\end{equation}
	where $\kappa_{B}>0$ depends on the coefficients $\mu$ and
	$\varepsilon$ but is independent of the mesh size.
\end{lemma}

Now, we are ready to prove the well-posedness of
\eqref{eq:MDG1}-\eqref{eq:MDG2} for $k=0$.

\begin{proof}
	[Proof of Theorem \ref{thm:indefUniqueness} for $k=0$] 
	From the classical theory of mixed FEM \cite[see,
	e.g.,][]{Boffi2013}, Lemma \ref{lem:continuity},
	\ref{lem:ellipticity} and \ref{lem:inf-sup} imply that the
	auxiliary formulation
	\eqref{eq:auxiliary1}-\eqref{eq:auxiliary2} with $k=0$ has a
	unique solution $(\bm{u}_h, \lambda_h; p_h)\in W_h \times
	Q_h$ and satisfies 
	\begin{equation}
		\label{eq:stability}
		\|(\bm{u}_{h},\lambda_{h})\|_{W(h)}+\|p_h\|_{Q(h)}
		\leq C \|\bm{j}\|_{0},
	\end{equation}
	where $C>0$ is independent of the mesh size. From Lemma
	\ref{lem:same solution}, $(\bm{u}_h, p_h)\in V_h\times Q_h$
	also solves \eqref{eq:MDG1}-\eqref{eq:MDG2} and the
	uniqueness of \eqref{eq:MDG1}-\eqref{eq:MDG2} follows from
	the a priori estimate \eqref{eq:stability}.
\end{proof}

\subsection{Conforming approximation}
In the error analysis, we shall use the conforming projection
$\Pi_h^{\mathrm{c}}$, introduced in \cite[Proposition
4.5]{Houston2005}, which states that the approximation of a
discontinuous function in $V_h$ by a $H(\mathrm{curl})$
averaging operator can be quantified in terms of certain jumps.
The following lemma is actually a byproduct of the proof of
\cite[Proposition 4.5]{Houston2005}, see
\cite[Appendix]{Houston2005} for more details.

\begin{lemma} [Conforming approximation]
	\label{lem:operator Pi}
	There exists an operator $\Pi_h^{\mathrm{c}}: V_{h}
	\rightarrow V_{h0}^{\mathrm{c}}$ such that for all
	$\bm{v}\in V_{h}$
	\begin{align*}
		h_K^{-2}\|\bm{v}-\Pi_h^{\mathrm{c}}\bm{v}\|_{0,K}^{2}
		+ \|\nabla\times (\bm{v}-\Pi_h^{\mathrm{c}}\bm{v})\|_{0,K}^{2}
		\leq C \sum_{F\in \partial K}\int_{F} h_F^{-1}
		|\llbracket \bm{v} \rrbracket_{T}|^{2} \,\mathrm{d}s,
	\end{align*}
	Here, the constant $C>0$ depends on the shape-regularity of
	the mesh and the polynomial degree $\ell$, but not on the
	mesh size.
\end{lemma}

By using the stability of the lifting operator \eqref{eq:lifting
stability}, Lemma \ref{lem:operator Pi} immediately implies the
following approximation and stability result.
\begin{lemma}
	\label{lem:operator Pi 2}
	The projection $\Pi_h^{\mathrm{c}}$ from Lemma
	\ref{lem:operator Pi} satisfies for all $\bm{v}\in V_{h}$
	\begin{align*}
		h_K^{-2}\|\bm{v}-\Pi_h^{\mathrm{c}}\bm{v}\|_{0,K}^{2}
		+ \|\nabla\times (\bm{v}-\Pi_h^{\mathrm{c}}\bm{v})\|_{0,K}^{2}
		\leq C \sum_{F\in \partial K} \|\mathcal{R}_F(
			\llbracket \bm{v}\rrbracket_{T})\|_{0,\Omega}^2.
	\end{align*}
	Here, the constant $C>0$ depends on the shape-regularity of
	the mesh and the polynomial degree $\ell$, but not on the
	mesh size.
\end{lemma}

\subsection{Smoothed interpolation}

The idea of combining the canonical finite element interpolation
operators with some mollification technique for low regularity
functions has been introduced in many papers, e.g., by
\cite{Schoberl2001, Schoberl2008}, \cite{Arnold2006},
\cite{Christiansen2008}, \cite{Falk2014} and \cite{Ern2016}. In
this section, we combine the shrinking-based mollification in
\cite[Section 3]{Ern2016} with canonical finite element
interpolations to prove the convergence of the DG approximation
to solutions of the Maxwell equations with low regularity
requirement.

To prove the local approximation properties, stated in
Proposition \ref{pop:Ih} and \ref{pop:Ih0} below, 
we will employ a family of smoothing
operators developed in \cite[Section 6.1]{Ern2016}.


Following \cite[]{Christiansen2008, Ern2016}, we define for
$\varrho \in (0,1)$ 
\begin{equation}
	\label{eq:delta def}
	\delta(\bm{x}):=\varrho g_h(\bm{x}), \quad \bm{x}\in \Omega,
\end{equation}
where $g_h(x)\in C^{0,1}(\bar{\Omega})$ is a mesh-size function
such that there are constants $c^{\prime},c^{\prime\prime}>0$
satisfying
\begin{equation}
	\label{eq:gh}
	c^{\prime}h_K\leq g_h(\bm{x})\leq c^{\prime\prime}h_K,
	\quad \forall \bm{x}\in K.
\end{equation}
Let $\mathcal{K}_{\delta}(\mathcal{K}_{\delta}^{\mathrm{g}},
\mathcal{K}_{\delta}^{\mathrm{c}},
\mathcal{K}_{\delta}^{\mathrm{d}},
\mathcal{K}_{\delta}^{\mathrm{b}}):L^1(\Omega)\to
C^1(\overline{\Omega})$ be the families of mollification
operators introduced in \cite[Section 3.2 and 5.2]{Ern2016}. Let
$\mathcal{I}_{h}(\mc{I}_h^{\mr{g}}, \mc{I}_h^{\mr{c}},
\mc{I}_h^{\mr{d}}, \mc{I}_h^{\mr{b}})$ be the canonical finite
element interpolation operators, i.e.,
$\mathcal{I}_{h}^{\mathrm{g}}$ the Lagrange interpolation,
$\mathcal{I}_{h}^{\mathrm{c}}$ the standard N\'{e}d\'{e}lec
interpolation of the first kind \cite[see][Section 5.5]{Monk2003},
$\mathcal{I}_{h}^{\mathrm{d}}$ the divergence conforming
interpolation \cite[see][Section 5.4]{Monk2003} and
$\mathcal{I}_{h}^{\mathrm{b}}$ the $L^2$ projection
\cite[see][Section 5.7]{Monk2003}, which enjoy a commuting
diagram property \cite[see][(5.59)]{Monk2003}.

Combining the mollification operators $\mc{K}_{\delta}$ with the
canonical interpolation operators, we obtain the smoothed
interpolation operators
\begin{equation}
	\label{eq:Ih}
	\widetilde{\mathcal{I}}_h:=\mathcal{I}_h \mathcal{K}_\delta.
\end{equation}
We note that the canonical interpolation operators require
sufficient smoothness while the smoothed operators requires
merely $L^1$-integrability. From the definition of the smoothed
interpolation and commuting diagrams
\cite[see][(5.59)]{Monk2003} and \cite[(3.7)]{Ern2016}, we
deduce the following lemma. 

\begin{lemma} [Commuting properties]
	\label{lem:diagram}
	Let $\wt{\mc{I}}_{h}$ be defined as in \eqref{eq:Ih}. There
	hold:
	\begin{enumerate}
		\item[$(\mr{i})$] $\nabla\times \wt{\mc{I}}_{h}^{\mr{c}} \bm{v} = \wt{\mc{I}}_{h}^{\mr{d}}\nabla\times  \bm{v}$, for all $\bm{v}\in H(\mathrm{curl},\Omega)$;
		\item[$(\mr{ii})$] $\nabla\cdot \wt{\mc{I}}_{h}^{\mr{d}} \bm{v} = \wt{\mc{I}}_{h}^{\mr{b}}\nabla\cdot  \bm{v}$, for all $\bm{v}\in H(\mathrm{div},\Omega)$.
	\end{enumerate}
\end{lemma}

We finish this section by a family of approximation results.
Since the proof is technical, especially some new local
properties of $\mathcal{K}_\delta$ are needed, we postpone it to
Appendix A.

\begin{proposition} [Local approximation]
	\label{pop:Ih}
	Let $\widetilde{\mathcal{I}}_h$ be defined as in
	\eqref{eq:Ih} and $s\in [0,\frac{1}{2})$. There exists a
	constant $c>0$ independent of the mesh size such that for
	all $\bm{v}\in H(\mr{curl}, \Omega)\cap
	H^s(\mathcal{T}_h)^d$
	\[
		\|\bm{v}-\widetilde{\mathcal{I}}_h^{\mathrm{c}}\bm{v}\|_{0,K} 
		\leq c \left( h_K^{s}\|\bm{v}\|_{s,D_K} 
		+h_K\|\nabla \times \bm{v}\|_{0,D_K} \right),
	\]
	where $D_K$ is the macro element defined in Theorem
	\ref{thm:defEnergyNorm}. Similarly, there hold
	\begin{alignat*}{3}
		\|\bm{v}-\widetilde{\mathcal{I}}_h^{\mathrm{d}}\bm{v}\|_{0,K} 
		&\leq c h_K^{s}\|\bm{v}\|_{s,D_K}, \qquad && \forall \bm{v}\in H(\mathrm{div},\Omega)\cap H^s(\mathcal{T}_h)^d,\\
		\|v-\widetilde{\mathcal{I}}_h^{\mathrm{b}}v\|_{0,K} 
		&\leq c	h_K^{s}\|v\|_{s,D_K}, && \forall v\in 
		H^s(\mathcal{T}_h).
	\end{alignat*}
\end{proposition}

\subsection{Smoothed interpolation with boundary conditions}

Similar to last section, we establish the approximation
properties of the smoothed interpolation with boundary
restriction, which will be used to prove the best approximation
given in Theorem \ref{thm:bestAppr}. 

Let $\mathcal{K}_{\delta, 0}(\mathcal{K}_{\delta, 0}^{\mathrm{g}},
\mathcal{K}_{\delta, 0}^{\mathrm{c}},
\mathcal{K}_{\delta, 0}^{\mathrm{d}},
\mathcal{K}_{\delta, 0}^{\mathrm{b}}):L^1(\Omega)\to
C_0^1(\overline{\Omega})$ be the families of mollification
operators introduced in \cite[Section 4.2]{Ern2016}.
Then the smoothed interpolations
\begin{equation}
	\label{eq:Ih0}
	\widetilde{\mathcal{I}}_{h0}:=
	\mathcal{I}_h \mathcal{K}_{\delta,0}
\end{equation}
satisfy the following commuting properties.

\begin{lemma} [Commuting properties]
	\label{lem:diagram 2}
	Let $\wt{\mc{I}}_{h0}$ be defined as in \eqref{eq:Ih0}. There
	hold:
	\begin{enumerate}
		\item[$(\mr{i})$] $\nabla\times \wt{\mc{I}}_{h0}^{\mr{c}} \bm{v} = \wt{\mc{I}}_{h0}^{\mr{d}}\nabla\times  \bm{v}$, for all $\bm{v}\in H_0(\mathrm{curl},\Omega)$;
		\item[$(\mr{ii})$] $\nabla\cdot \wt{\mc{I}}_{h0}^{\mr{d}} \bm{v} = \wt{\mc{I}}_{h0}^{\mr{b}}\nabla\cdot  \bm{v}$, for all $\bm{v}\in H_0(\mathrm{div},\Omega)$.
	\end{enumerate}
\end{lemma}



By following the proof of Proposition \ref{pop:Ih}, we can
conclude the following approximation results.

\begin{proposition} [Local approximation]
	\label{pop:Ih0}
	Let $\widetilde{\mathcal{I}}_{h0}$ be defined as in
	\eqref{eq:Ih0} and $s\in [0,\frac{1}{2})$. There exists a
	constant $c>0$ independent of the mesh size such that for
	all $\bm{v}\in H_0(\mr{curl}, \Omega)\cap
	H^s(\mathcal{T}_h)^d$
	\[
		\|\bm{v}-\widetilde{\mathcal{I}}_{h0}^{\mathrm{c}}\bm{v}\|_{0,K} 
		\leq c \left( h_K^{s}\|\bm{v}\|_{s,D_K} 
		+h_K\|\nabla \times \bm{v}\|_{0,D_K} \right).
	\]
	Similarly, there holds
	\begin{alignat*}{3}
		\|\bm{v}-\wt{\mc{I}}_{h0}^{\mr{d}}\bm{v}\|_{0,K} 
		&\leq c h_K^{s}\|\bm{v}\|_{s,D_K}, \quad 
		& \forall \bm{v}\in H_{0}(\mr{div},\Omega)\cap H^s(\mc{T}_h)^{d}.
	\end{alignat*}
\end{proposition}

\section{Definite Maxwell equations}
\label{sec:definiteMax}

The error estimates of the mixed DG discretization \eqref{eq:MDG1}-\eqref{eq:MDG2} for the definite Maxwell
equations with $k=0$ will greatly facilitate the analysis for the indefinite problem discussed in Section~\ref{sec:well-posedness}.

\subsection{Residual operators}

Following \cite[Section 6.1]{Houston2004}, we introduce two
consistency-related residual operators, which play a key role in
deriving an \textit{a priori} error estimate under minimal smoothness
requirements. 

Suppose that $(\bm{u},p)\in V\times Q$ is the exact solution of
continuous variational problem \eqref{eq:conForm1}-\eqref{eq:conForm2}. We define the residuals
\[
	\mathcal{R}_1(\bm{u},p;\bm{v},\eta):=
		A_h(\bm{u},0;\bm{v},\eta)+B_h(\bm{v},\eta;p)-(\bm{j},\bm{v}),
		\quad \mbox{and}\quad
		\mathcal{R}_2(\bm{u},q)=B_h(\bm{u},0;q),
\]
for all $(\bm{v},\eta)\in W_h$ and $q\in Q_h$. We also define
norms of the residual operators as
\[
	\mathcal{R}_1(\bm{u},p)=\sup_{0\neq (\bm{v},\eta)\in W_h} 
		\frac{\mathcal{R}_1(\bm{u},p;\bm{v},\eta)}{\|(\bm{v},\eta)\|_{W(h)}},
	\quad
	\mathcal{R}_2(\bm{u}) = \sup_{0\neq q\in Q_h} 
		\frac{\mathcal{R}_2(\bm{u};q)}{\|q\|_{Q(h)}}.
\]

The analysis of \cite{Houston2004} relies crucially on the the
smoothness $H^s(\mathcal{T}_h), s>\frac{1}{2}$. In this section,
we will extend the analysis to $s\in [0,\frac{1}{2})$.

\begin{lemma}
	\label{lem:consistency Vhc}
	Let $(\bm{u},p)\in V\times Q$ be the solution of
	\eqref{eq:conForm1}-\eqref{eq:conForm2}, then
	\begin{alignat*}{3}
		\mathcal{R}_1(\bm{u},p;\bm{v},\eta)&=0, \qquad
			&& \forall (\bm{v},\eta)\in V_{h0}^{\mathrm{c}}\times M_h,\\
		\mathcal{R}_2(\bm{u},q) &=0,	&& \forall q\in Q_h^{\mathrm{c}}.
	\end{alignat*}
\end{lemma}

\begin{proof}
	The identities are actually direct results of
	\eqref{eq:conForm1}-\eqref{eq:conForm2}. In fact, for all $\bm{v}\in
	V_{h0}^{\mathrm{c}}$, there holds from the definition of
	$\mathcal{R}_1$ and
	\eqref{eq:conForm1}
	\[
		\mathcal{R}_1(\bm{u},p;\bm{v},\eta)=
		a(\bm{u},\bm{v})+b(\bm{v},q)-(\bm{j},\bm{v})=0.
	\]
	The other identity follows from \eqref{eq:conForm2}
	\[
		\mathcal{R}_2(\bm{u},q) = b(\bm{u},q)=0.
	\]
\end{proof}

The following lemma estimates the residuals for the smooth case
$s>\frac{1}{2}$.

\begin{lemma}[Residual estimates]
	\label{lem:res est}
	Let $(\bm{u},p)\in V\times Q$ be the solution of
	\eqref{eq:conForm1}-\eqref{eq:conForm2} with
	\[
		\varepsilon\bm{u},\ \mu^{-1}\nabla\times \bm{u} 
		\in H^s(\mc{T}_h),\ s> \frac{1}{2}.
	\]
	Then, there hold
	\begin{align*}
		\mathcal{R}_1(\bm{u},p)
		&\leq  C\left( \sum_{K\in \mc{T}_h} h_K^{2\min\{s,\ell\}}
		\|\mu^{-1}\nabla\times \bm{u}\|_{s,K}^2 \right)^{1/2}, \\
		\mathcal{R}_2(\bm{u})
		&\leq C\left( \sum_{K\in \mc{T}_h} h_K^{2\min\{s,\ell\}}
		\|\varepsilon \bm{u}\|_{s,K}^2 \right)^{1/2} ,
	\end{align*}
	where the constant $C>0$ is independent of the mesh size.
\end{lemma}
\begin{proof}
	See the proof of \cite[Proposition 6.2]{Houston2004}.
\end{proof}


Now, we are ready to state our main results about the residuals. 

\begin{proposition}[Residual estimates]
	\label{pop:res est}
	Let $(\bm{u},p)\in V\times Q$ be the solution of
	\eqref{eq:conForm1}-\eqref{eq:conForm2} for $k=0$ with
	\[
		\varepsilon\bm{u},\ \mu^{-1}\nabla\times \bm{u} 
		\in H^s(\mc{T}_h), \ 0\leq s< \frac{1}{2},
	\]
	Then, there hold
	\begin{align*}
		\mathcal{R}_1(\bm{u},p)
		&\leq  C\left( \sum_{K\in \mc{T}_h} h_K^{2s} \|\mu^{-1}\nabla\times \bm{u}\|_{s,D_K}^2
		+ h_{K}^2\|\nabla\times \mu^{-1}\nabla\times \bm{u}\|_{0,D_K}^2
		\right)^{1/2}, \\
		\mathcal{R}_2(\bm{u})
		&\leq  C\left( \sum_{K\in \mc{T}_h} h_K^{2s} 
			\|\varepsilon \bm{u}\|_{s,D_K}^2 \right)^{\frac{1}{2}},
	\end{align*}
	where the constant $C>0$ is independent of the mesh size.
\end{proposition}

\begin{proof}
	Since $\bm{u}\in V$ we have $\llbracket
	\bm{u}\rrbracket_{T}=\bm{0}$,  and from \eqref{eq:conForm2}
	we have $\nabla\cdot \varepsilon\bm{u}=0$, which implies
	$\llbracket \varepsilon\bm{u}\rrbracket_{N}=0$. Let
	$\sigma(\bm{u}) := \mu^{-1}\nabla\times \bm{u}$. From
	\eqref{eq:conForm1}, we know that 
	\begin{equation}
		\label{eq:lem:expression 0}
		\nabla\times \sigma(\bm{u})-\varepsilon \nabla p 
		= \bm{j}\in L^2(\Omega).
	\end{equation}
	Hence, $\sigma(\bm{u})\in H(\mathrm{curl},\Omega)$ by
	noticing that $\nabla p\in L^2(\Omega)$ and that
	$\varepsilon$ bounded. 
	
	\textit{Step 1: Estimate of $\mc{R}_1$}.  For any
	$(\bm{v},\eta)\in W_h$, from the definition of
	$\mathcal{R}_1$ and \eqref{eq:lem:expression 0}, one gets 
	\begin{equation}
		\label{eq:lem:expression 1}
		\begin{split}
			\mathcal{R}_1(\bm{u},p;\bm{v},\eta)
			&=(\sigma(\bm{u}), \nabla_h\times \bm{v})
			-\left(  \mathcal{R}(\llbracket \bm{v}\rrbracket_{T}), 
			\sigma(\bm{u}) \right)
			-(\varepsilon\bm{v},\nabla p)-(\bm{j},\bm{v}) \\
			&=-\left( \nabla\times \sigma(\bm{u}), \bm{v} \right)
			+(\sigma(\bm{u}), \nabla_h\times \bm{v})
			-\left( \mathcal{R}(\llbracket \bm{v}\rrbracket_{T}), 
			\sigma(\bm{u}) \right).
		\end{split}
	\end{equation}
	To treat the last term, we employ the smoothed
	interpolation $\wt{\mc{I}}_h^{\mr{c}}: H(\mr{curl},\Omega)
	\to V_h^{\mr{c}}$ from \eqref{eq:Ih}, i.e.,
	\begin{equation*}
	\begin{split}
		\left(\mathcal{R}(\llbracket \bm{v}\rrbracket_{T}), 
			\sigma(\bm{u}) \right) 
		&= \left(\mathcal{R}(\llbracket \bm{v}\rrbracket_{T}), 
		\wt{\mc{I}}_h^{\mr{c}}\sigma(\bm{u}) \right)
		+\left(\mathcal{R}(\llbracket \bm{v}\rrbracket_{T}), 
			\sigma(\bm{u})-\wt{\mc{I}}_h^{\mr{c}} \sigma(\bm{u}) \right).
	\end{split}
	\end{equation*}
	By using the definition of the lifting operator $\mc{R}$ and
	integration by parts, we infer for $v\in V_h$
	\begin{equation*}
	\begin{split}
		\left(\mathcal{R}(\llbracket \bm{v}\rrbracket_{T}), 
		\wt{\mc{I}}_h^{\mr{c}} \sigma(\bm{u}) \right)
		&= \sum_{F\in \mathcal{F}_h}\int_{F} \llbracket \bm{v}\rrbracket_{T}
		\cdot\llbrace \wt{\mc{I}}_h^{\mr{c}} \sigma(\bm{u})\rrbrace \,\mathrm{d}s \\
		&= -(\nabla\times \wt{\mc{I}}_h^{\mr{c}} \sigma(\bm{u}), \bm{v})
		+(\wt{\mc{I}}_h^{\mr{c}} \sigma(\bm{u}), \nabla_h\times \bm{v}),
	\end{split}
	\end{equation*}
	and obtain
	\begin{equation*}
		\left(\mathcal{R}(\llbracket \bm{v}\rrbracket_{T}), 
		\sigma(\bm{u}) \right) 
		= -(\nabla\times \wt{\mc{I}}_h^{\mr{c}} \sigma(\bm{u}), \bm{v})
		+(\wt{\mc{I}}_h^{\mr{c}} \sigma(\bm{u}), \nabla_h\times \bm{v})
		+(\mathcal{R}(\llbracket \bm{v}\rrbracket_{T}), 
		\sigma(\bm{u})-\wt{\mc{I}}_h^{\mr{c}} \sigma(\bm{u}) ).
	\end{equation*}
	Substituting this identity into
	\eqref{eq:lem:expression 1}, we have 
	\begin{equation}
	\label{eq:lem:expression 2}
	\begin{split}
		\mathcal{R}_1(\bm{u},p;\bm{v},\eta)
		&=-\left( \nabla\times \sigma(\bm{u}) - \nabla\times \wt{\mc{I}}_h^{\mr{c}} \sigma(\bm{u}), \bm{v} \right)
		+ (\sigma(\bm{u})-\wt{\mc{I}}_h^{\mr{c}} \sigma(\bm{u}), \nabla_h\times \bm{v}) \\
		&-\left( \mathcal{R}(\llbracket \bm{v}\rrbracket_{T}), 
		\sigma(\bm{u}) - \wt{\mc{I}}_h^{\mr{c}} \sigma(\bm{u}) \right).
	\end{split}
	\end{equation}

	Let $\bm{v}^{\perp}=\bm{v}-\Pi_h^{\mathrm{c}}\bm{v}$, then
	clearly $\llbracket \bm{v}\rrbracket_{T} = \llbracket
	\bm{v}^{\perp} \rrbracket_{T}$. Since
	$\Pi_h^{\mathrm{c}}\bm{v}\in V_{h0}^{\mathrm{c}}\subset
	V_h$, we know that \eqref{eq:lem:expression 2} also holds
	with $\bm{v}$ replaced by $\Pi_{h}^{\mathrm{c}}\bm{v}$. From
	Lemma \ref{lem:consistency Vhc}, we know
	$\mathcal{R}_1(\bm{u},p; \Pi_h^{\mathrm{c}}\bm{v},\eta) =
	0$, hence 
	\begin{equation*}
	\begin{split}
		\mathcal{R}_1 (\bm{u} & ,p;\bm{v},\eta)
	=\mathcal{R}_1(\bm{u},p;\bm{v}-\Pi_h^{\mathrm{c}}\bm{v},0)	\\
	=&	-( \nabla\times \sigma(\bm{u}) - \nabla\times 
		\wt{\mc{I}}_h^{\mr{c}} \sigma(\bm{u}), \bm{v}^{\perp} ) 
		+ (\sigma(\bm{u})-\wt{\mc{I}}_h^{\mr{c}} 
		\sigma(\bm{u}), \nabla_h\times \bm{v}^{\perp})		\\
		&-( \mathcal{R}(\llbracket \bm{v}^{\perp}\rrbracket_{T}), 
		\sigma(\bm{u}) - \wt{\mc{I}}_h^{\mr{c}} \sigma(\bm{u}) ) \\
	\leq & \sum_{K\in \mc{T}_h} \left( \|\nabla\times \sigma(\bm{u}) 	- \nabla\times \wt{\mc{I}}_h^{\mr{c}} 
		\sigma(\bm{u})\|_{0,K} 	\|\bm{v}^{\perp}\|_{0,K}
		+\|\sigma(\bm{u}) - \wt{\mc{I}}_h^{\mr{c}} \sigma(\bm{u})
		\|_{0,K}\|\nabla\times \bm{v}^{\perp}\|_{0,K} \right) \\
	&	+\sum_{F\in \mc{F}_h} \|\sigma(\bm{u}) - \wt{\mc{I}}_h^{\mr{c}} \sigma(\bm{u})\|_{0,\omega_F} 
		\|\mathcal{R}_F(\llbracket \bm{v}^{\perp}\rrbracket_{T})\|_{0,\omega_F}  	\\
	\leq & C\left(h_K^2 \|\nabla\times \sigma(\bm{u}) - 
		\wt{\mc{I}}_h^{\mr{d}} \nabla\times  \sigma(\bm{u})\|_{0,K}^2
		+ \sum_{K\in \mc{T}_h} \|\sigma(\bm{u}) - 
		\wt{\mc{I}}_h^{\mr{c}} \sigma(\bm{u})\|_{0,K}^2 \right)^{\frac{1}{2}}
		\left( \sum_{F\in \mc{F}_h}\|\mathcal{R}_F(\llbracket \bm{v}\rrbracket_{T})\|_{0,\Omega} \right)^{\frac{1}{2}},
	\end{split}
	\end{equation*}
	where we have used the conforming approximation in Lemma
	\ref{lem:operator Pi 2} and the commuting property (i) of
	Lemma \ref{lem:diagram}. Finally, the estimate for
	$\mc{R}_{1}$ follows directly from Proposition \ref{pop:Ih}.

	\textit{Step 2: Estimate of $\mc{R}_2$}. 
	Similarly, we have for $q\in Q_h$
	\begin{equation}
		\label{eq:lem:expression 3}
		\begin{split}
			\mathcal{R}_2(\bm{u};q)
			&=-(\varepsilon \bm{u}, \nabla_h q)
			+\left( \mathcal{R}(\llbracket q\rrbracket_{N}), 
				\varepsilon \bm{u} 	\right)	\\
			&=-(\varepsilon \bm{u}, \nabla_h q)
			+\left( \mathcal{R}(\llbracket q\rrbracket_{N}), 	
			\wt{\mc{I}}_h^{\mathrm{d}} \varepsilon \bm{u} 	\right)
			+\left( \mathcal{R}(\llbracket q\rrbracket_{N}), 
				\varepsilon \bm{u} - \wt{\mc{I}}_h^{\mathrm{d}} \varepsilon \bm{u} \right).
		\end{split}
	\end{equation}
	Since $\wt{\mc{I}}_h^{\mathrm{d}} \varepsilon \bm{u}$
	belongs to the divergence conforming finite element space
	$\widetilde{V}_h^{\mathrm{d}}$ and
	$\widetilde{V}_h^{\mathrm{d}} \subset V_h$, using the
	definition of lifting operator $\mc{R}$ and integration by
	parts, we obtain
	\begin{equation}
		\label{eq:lem:expression 4}
		\left( \mathcal{R}(\llbracket q\rrbracket_{N}), 	
		\wt{\mc{I}}_h^{\mathrm{d}} \varepsilon \bm{u} 	\right)
		= \int_{\mathcal{F}_h} \llbracket q\rrbracket_{N} \cdot
		\llbrace \wt{\mc{I}}_h^{\mathrm{d}} \varepsilon \bm{u}\rrbrace \,\mathrm{d}s
		= (\nabla\cdot \wt{\mc{I}}_h^{\mathrm{d}} \varepsilon \bm{u}, q )
		+ (\wt{\mc{I}}_h^{\mathrm{d}} \varepsilon \bm{u}, \nabla_h q).
	\end{equation}
	From the commuting property (ii) of Lemma
	\ref{lem:diagram} and $ \nabla\cdot \varepsilon \bm{u}=0$,
	one arrives at
	\begin{equation}
		\label{eq:lem:expression 5}
		\nabla\cdot \wt{\mc{I}}_h^{\mathrm{d}} \varepsilon \bm{u}
		=\wt{\mc{I}}_h^{\mathrm{b}} \nabla\cdot \varepsilon \bm{u} = 0.
	\end{equation}
	Substituting \eqref{eq:lem:expression
	4}-\eqref{eq:lem:expression 5} into \eqref{eq:lem:expression
	3} leads to 
	\begin{equation*}
	\begin{split}
		\mathcal{R}_2(\bm{u};q)
		&=-(\varepsilon \bm{u} - \wt{\mc{I}}_h^{\mathrm{d}} \varepsilon \bm{u}, \nabla_h q)
		+\left( \mathcal{R}(\llbracket q\rrbracket_{N}), 
			\varepsilon \bm{u} - \wt{\mc{I}}_h^{\mathrm{d}} \varepsilon \bm{u} \right) \\
		&\leq C \left( \sum_{K\in \mc{T}_h} 
		\|\varepsilon \bm{u} - \wt{\mc{I}}_h^{\mathrm{d}} \varepsilon \bm{u}\|_{0,K}^2 \right)^{1/2} \|q\|_{Q(h)}, 
	\end{split}
	\end{equation*}
	where the estimate for $\mc{R}_2$ is a direct result of
	Proposition \ref{pop:Ih}.
\end{proof}


Combining Lemma \ref{lem:res est} and Proposition \ref{pop:res
est}, we obtain the following residual estimates.

\begin{corollary}
	\label{cor:res est}
	For $k=0$, let $(\bm{u},p)\in V\times Q$ be the solution of
	\eqref{eq:conForm1}-\eqref{eq:conForm2} with
	\[
		\varepsilon\bm{u},\ \mu^{-1}\nabla\times \bm{u} 
		\in H^s(\mc{T}_h), \ s\geq 0.
	\]
	Then, there hold
	\begin{align*}
		\mathcal{R}_1(\bm{u},p)
		&\leq  C\left( \sum_{K\in \mc{T}_h} h_K^{2\min\{s,\ell\}} \|\mu^{-1}\nabla\times \bm{u}\|_{s,D_K}^2
		+ \chi(s)h_{K}^2\|\nabla\times \mu^{-1}\nabla\times \bm{u}\|_{0,D_K}^2
		\right)^{\frac{1}{2}}, \\
		\mathcal{R}_2(\bm{u})
		&\leq  C\left( \sum_{K\in \mc{T}_h} h_K^{2\min\{s,\ell\}} 
			\|\varepsilon \bm{u}\|_{s,D_K}^2 \right)^{\frac{1}{2}},
	\end{align*}
	where the constant $C>0$ is independent of the mesh size,
	$\chi(s)=1$ if $s\leq \frac{1}{2}$ and $\chi(s)=0$
	otherwise. Note that $D_K$ in the estimates can be replaced
	by $K$ for $s>\frac{1}{2}$.
\end{corollary}

\subsection{Error estimates}

The framework for the abstract error estimates, combined with
the stability conditions given in Section \ref{subs:auxiliary},
leads to the abstract error estimates for the auxiliary
formulation \eqref{eq:auxiliary1}-\eqref{eq:auxiliary2} with
$k=0$.

\begin{theorem} 
\label{thm:defAbstractError}
	For $k=0$, let $(\bm{u},p)$ be the solution of the weak
	formulation of the time-harmonic Maxwell equations
	\eqref{eq:conForm1}-\eqref{eq:conForm2}, and let
	$(\bm{u}_{h}, \lambda_{h}; p_{h})$ be the solution of
	discontinuous Galerkin discretization
	\eqref{eq:auxiliary1}-\eqref{eq:auxiliary2}. Then, there
	exists a constant $C>0$, independent of mesh size, such that
	\begin{align*}
		\|(\bm{u}-\bm{u}_{h},-\lambda_{h})\|_{W(h)} 
		\leq  C \Big( &\inf_{\bm{v}\in V_{h}}\|\bm{u}-\bm{v}\|_{V(h)}
			+ \inf_{q\in Q_h}\|p-q\|_{Q(h)}\\
			&+ \mathcal{R}_1(\bm{u},p)+ \mathcal{R}_2(\bm{u})
			\Big), \\
		\|p-p_{h}\|_{Q(h)} \leq  C \Big( &
			\inf_{q \in Q_{h}}\|p-q\|_{Q(h)}
			+\|(\bm{u}-\bm{u}_{h},-\lambda_{h})\|_{W(h)}
			+ \mathcal{R}_1(\bm{u},p) \Big).
	\end{align*}
\end{theorem}
\begin{proof}
	The estimates follows from a extension of the standard mixed
	finite theory, see the proof of \cite[Theorem
	6.1]{Houston2004} combined with the stability conditions in
	Lemma \ref{lem:continuity}-\ref{lem:inf-sup}.
\end{proof}

The next results quantifies the best-approximation error.

\begin{theorem}
	\label{thm:bestAppr}
	Suppose that $\bm{u}\in H_{0}(\mathrm{curl},\Omega)$ such
	that $\varepsilon\bm{u}, \mu^{-1}\nabla\times \bm{u} \in
	H^s(\mathcal{T}_h)$ and $p\in Q$ such that $\nabla p\in
	H^{s}(\mathcal{T}_h)$, $s\geq 0$. There exists a constant
	$C>0$ independent of the mesh size such that 
	\begin{gather} 
		\inf_{\bm{v}\in \bm{V}_{h}}\|\bm{u}-\bm{v}\|_{V(h)}
		\leq  C\left(\sum_{K\in \mathcal{T}_h} h_K^{2\min\{s,\ell\}}
		\left(	\|\varepsilon\bm{u}\|_{s,D_K}^2	+
		\|\mu^{-1}\nabla\times\bm{u}\|_{s,D_K}^2 \right) \right)^{1/2},	
		\label{eq:bestAppr 1}\\
		\inf_{q\in Q_h}\|p-q\|_{Q(h)} \leq C \left(\sum_{K\in \mathcal{T}_h} h_K^{2\min\{s,\ell\}} \|p\|_{1+s,D_K}^2 \right)^{1/2}.
		\label{eq:bestAppr 2}
	\end{gather}
\end{theorem}


\begin{proof}
	By taking $\bm{v}=\widetilde{\mc{I}}_{h0}^{\mr{c}} \bm{u}$,
	which is in $H_{0}(\mathrm{curl},\Omega)$, and using
	commuting property (i) of Lemma \ref{lem:diagram 2}, we
	have
	\begin{equation*}
	\begin{split}
		\inf_{\bm{v}\in V_h}\|\bm{u}-\bm{v}\|_{V(h)}
		&\leq \left( \|\varepsilon^{\frac{1}{2}}(\bm{u}
			- \widetilde{\mc{I}}_{h0}^{\mr{c}}\bm{u})\|_{0,\Omega}^{2}
		+ \|\mu^{-\frac{1}{2}}\nabla\times (\bm{u}
		- \widetilde{\mc{I}}_{h0}^{\mr{c}}\bm{u})\|_{0,\Omega}^{2} \right)^{1/2}\\
		&=\left( \|\varepsilon^{\frac{1}{2}}(\bm{u}
		- \widetilde{\mc{I}}_{h0}^{\mr{c}}\bm{u})\|_{0,\Omega}^{2}
		+ \|\mu^{-\frac{1}{2}}(\nabla\times \bm{u}
		- \widetilde{\mc{I}}_{h0}^{\mr{d}} \nabla\times\bm{u})\|_{0,\Omega}^{2} \right)^{1/2} .
	\end{split}
	\end{equation*}
	Then, \eqref{eq:bestAppr 1} follows from the approximation
	properties given in Proposition \ref{pop:Ih0}.

	We note \eqref{eq:bestAppr 2} is the standard approximation
	result of Cl\'{e}ment interpolation \cite[Section
	5.6.1]{Monk2003}, or the Scott–Zhang interpolation
	\cite[see, e.g.,][]{Scott1990, Brenner2008}.
\end{proof}

Now, we are ready to prove Theorem \ref{thm:defEnergyNorm} in the special case of $k=0$.

\begin{proof}[Proof of Theorem \ref{thm:defEnergyNorm} for $k=0$] 
	From the abstract error estimates in Theorem
	\ref{thm:defAbstractError}, the polynomial approximation
	results \eqref{eq:bestAppr 1}-\eqref{eq:bestAppr 2} and
	the residual estimates in Corollary \ref{cor:res est}, we easily
	derive the a priori error bounds in Theorem
	\ref{thm:defEnergyNorm} for $k=0$.
\end{proof}

\section{Indefinite Maxwell equations}
\label{sec:well-posedness}

We will first discuss existence and
uniqueness properties of the mixed DG method
\eqref{eq:MDG1}-\eqref{eq:MDG2} for the indefinite Maxwell
equations \eqref{eq:Max-1}, and subsequently provide error
estimates for $k\neq 0$, $k^2$ not a Maxwell eigenvalue, under
minimal regularity requirements. Instead of using Fredholm
alternatives to show well-posedness of the mixed DG system
\eqref{eq:MDG1}-\eqref{eq:MDG2}, we shall prove an inf-sup
condition for $A_h$ on the kernel of $B_h$ for $k\neq 0$.

\subsection{Uniform convergence and spectral theory}
\label{sec:spectralTheory}

The spectral properties of the solution
operator are essential for the proof of the existence and
uniqueness of the mixed DG method for the time-harmonic Maxwell
equations. 
Because of the existence of the unique solution $(\bm{u},p)$ to
\eqref{eq:conForm1}-\eqref{eq:conForm2} for the Maxwell
equations with $k=0$ and the a priori estimate given in Lemma
\ref{lem:uniqueness}, we can uniquely define the bounded
solution operators $T:L_{\varepsilon}^{2}(\Omega)^{3}\rightarrow
V$ and $T_p:L_{\varepsilon}^{2}(\Omega)^{3}\rightarrow Q$ for
\eqref{eq:conForm1}-\eqref{eq:conForm2} by 
\begin{equation}
	\label{eq:TandTp}
	T\bm{j}:=\bm{u},\quad T_p\bm{j}:=p.
\end{equation}
Similarly, from the uniqueness of the solution $(\bm{u}_h,p_h)$
to \eqref{eq:MDG1}-\eqref{eq:MDG2} for the Maxwell equations
\eqref{eq:Max} with $k=0$ and the a priori estimate
\eqref{eq:stability}, we can uniquely define the bounded
discrete solution operators
$T_h:L^{2}_{\varepsilon}(\Omega)^{3}\rightarrow V_h$ and
$T_{p,h}:L^{2}_{\varepsilon}(\Omega)^{3}\rightarrow Q_h$ by 
\begin{equation}
	\label{eq:ThandTph}
	T_h\bm{j}:=\bm{u}_h,\quad T_{p,h}\bm{j}:=p_h.
\end{equation}


From the abstract estimates in Theorem
\ref{thm:defAbstractError} together with Corollary \ref{cor:res
est} and the consistency of the finite element spaces
\[
	\lim_{h\to 0}\inf_{\bm{v}\in V_h}\|\bm{u}-\bm{v}\|_{V(h)}=0,\quad 
	\lim_{h\to 0}\inf_{q\in Q_h}\|p-q\|_{Q(h)}=0,\quad \forall \bm{u}
	\in V,\ p\in Q,
\]
we obtain pointwise convergence of $T_h$ to $T$:  for any fixed $\bm{j}\in
L^2(\Omega)^d$, $\|T\bm{j}-T_h\bm{j}\|_{V(h)}\to 0$ as $h\to 0$.
The following proposition states the uniform convergence of
$T_h$, see Appendix C
for a proof.

\begin{proposition}[Uniform convergence]
	\label{pop:uni conv}
	\begin{equation}
		\label{eq:Th convergence}
		\lim_{h\to 0}\|T-T_h\|_{\mc{L}(V_h\to V(h))}=0.
	\end{equation}
\end{proposition}

By using \eqref{eq:Th convergence}, one can derive that for any
$z\in \rho(T)$, the resolvent set of $T$, the resolvent operator
$R_z(T_h)=(z-T_{h})^{-1}:V_{h}\rightarrow V_{h}$ exists and is
bounded for all $h$ sufficiently small. In fact, we have the
following lemma \cite[Lemma 1]{Descloux1978}.

\begin{lemma} \label{lem:Rz}
	Let \eqref{eq:Th convergence} hold and let $F \subset \rho(T)$ be
	closed. Then, for all $h$ small enough, there holds
	\[
		\|R_{z}(T_{h})\|_{\mathcal{L}(V_h)}
		=\sup_{\bm{v}\in V_h,\,\|\bm{v}\|_{V(h)}=1}
		\|R_{z}(T_{h})\bm{v}\|_{V(h)}
		\leq C,	\quad \forall\, z \in F,
	\]
	where $C>0$ depending on $F$ is independent of the mesh size
	$h$.
\end{lemma}

\subsection{Existence and uniqueness}

By using the results of Lemma \ref{lem:Rz}, it is
straightforward to prove Theorem \ref{thm:indefUniqueness}.

\begin{proof}[Proof of Theorem \ref{thm:indefUniqueness}] 
	We prove the uniqueness by proving the a priori bound in the
	theorem.
	It is obvious that there exists a unique element
	$\bm{j}_{h}$ such that 
	\[
		(\bm{j}_{h},\bm{v})_{\varepsilon}=(\bm{j},\bm{v}),
		\quad \forall\, \bm{v}\in V_{h}.
	\]
	Hence, we can rewrite \eqref{eq:MDG1}-\eqref{eq:MDG2} as
	follows:
	\begin{alignat*}{3}
		a_{h}(\bm{u}_{h},\bm{v})+b_{h}(\bm{v},p_{h}) 
		&= (\bm{j}_{h},\bm{v})_{\varepsilon} +k^{2} (\bm{u}_{h}, 
			\bm{v})_{\varepsilon}, \qquad && \forall\bm{v}\in V_h, \\
		b_{h}(\bm{u}_{h},q)-c_{h}(p_{h},q) 
		&= 0, 	&& \forall q\in Q_h.
	\end{alignat*}
	From the definition of the solution operators $T_{h}$ and 
	$T_{p,h}$, we infer that 
	\begin{equation} \label{eq:thm:indefUniqueness-1}
		\bm{u}_{h}=T_{h}\bm{j}_{h}+k^{2}T_{h}\bm{u}_{h}, \quad 
		p_{h}=T_{p,h}\bm{j}_{h}+k^{2}T_{p,h}\bm{u}_{h}. 
	\end{equation}
	With $z:=\frac{1}{k^{2}}$, the first equation becomes
	\[
		(z-T_{h})\bm{u}_{h}=zT_{h}\bm{j}_{h}.
	\]
	Since $k^{2}$ is not a Maxwell eigenvalue by assumption,
	i.e., $z$ is not an eigenvalue of $T$, Lemma \ref{lem:Rz}
	shows that, for $h$ small enough, the	operator
	$R_{z}(T_{h}):V_{h}\rightarrow V_{h}$ exists and is bounded
	uniformly in $h$. Hence, $\bm{u}_{h}$ is uniquely determined
	by 
	\[
		\bm{u}_{h}=z(z-T_{h})^{-1}T_{h}\bm{j}_{h}.
	\]
	From definition of $T_h$ \eqref{eq:ThandTph}, It follows
	that $T_h:L^2_{\varepsilon}(\Omega)^3\to V_h$ is also
	bounded and there exists a constant $C>0$ independent of the
	mesh size such that 
	\begin{equation} \label{eq:uh-bounded}
		\|\bm{u}_{h}\|_{V(h)}\leq C \|R_{z}(T_{h})\| \|T_h\|
			\|\varepsilon^{\frac{1}{2}} \bm{j}_{h}\|_{0,\Omega}
		\leq C \|\bm{j}\|_{0,\Omega}.
	\end{equation}
	The uniqueness of $p_{h}$ directly follows from 
	the uniqueness of $\bm{u}_{h}$, and there exists a constant $C>0$
	independent of mesh size such that 
	\begin{equation} \label{eq:ph-bounded}
		\|p_{h}\|_{Q(h)}\leq C \left( \|\varepsilon^{\frac{1}{2}} 
		\bm{j}_{h}\|_{0,\Omega}+ 
		\|\varepsilon^{\frac{1}{2}}\bm{u}_{h}\|_{0,\Omega} \right){}
		\leq C \|\bm{j}\|_{0,\Omega}.
	\end{equation}
\end{proof}

\subsection{Error estimates}
\label{subs:error est}

For the indefinite Maxwell equations ($k\neq 0$), we can similarly
define the residuals
\begin{align*}
	\mathcal{R}_{1,k}(\bm{u},p;\bm{v},\eta) &:=
		A_h(\bm{u},0;\bm{v},\eta)-k^{2}(\varepsilon\bm{u},\bm{v})
		+B_h(\bm{v},\eta;p)-(\bm{j},\bm{v}),\\
	\mathcal{R}_{2,k}(\bm{u},q) & := B_h(\bm{u},0;q).
\end{align*}
In the same way as in Section \ref{sec:definiteMax}, one can
show that the estimates for the residuals in Proposition
\ref{pop:res est} still hold true.

The next proposition provides the inf-sup condition of $A_{h}$
on the kernel of $B_{h}$ for $k\neq 0$, which is a crucial
ingredient for the error estimates. The idea of the proof is classical \cite[see, e.g.,][]{Melenk1995, Buffa2006}.

\begin{proposition} [Inf-sup condition]
	\label{pop:inf-sup Ah}
	For a mesh size $h$ small enough, there holds  
	\begin{equation}
		\label{eq:inf-sup-1} 
		\inf_{0\neq (\bm{u}_{h},\lambda_{h})\in \mr{Ker}(B_h)}
		\sup_{0\neq (\bm{v},\eta)\in \mr{Ker}(B_h)}
		\frac{A_{h}(\bm{u}_{h},\lambda_{h};\bm{v},\eta) 
		-k^{2}(\bm{u}_{h},\bm{v})_{\varepsilon}}
		{\|(\bm{u}_{h},\lambda_{h})\|_{W(h)}
		\|(\bm{v},\eta)\|_{W(h)}}\geq  \kappa_{A},
	\end{equation}
	for a positive constant $\kappa_{A}$, which depends on the
	coefficients $\mu$ and $\varepsilon$ but independent of $h$.
	Here, $\mr{Ker}(B_h)$ is the kernel of $B_h$ defined in
	\eqref{eq:Zh}.
\end{proposition}
\begin{proof}
	For any $(\bm{v},\eta)\in \mr{Ker}(B_h)$, let
	$(\bm{u}_{h},\lambda_{h}) = (\bm{v},\eta) +
	k^{2}(\tilde{\bm{u}}_{h}, \tilde{\lambda}_{h})$ with
	$(\tilde{\bm{u}}_{h}, \tilde{\lambda}_{h};\tilde{p}_{h})$ be
	the solution of DG method
	\eqref{eq:auxiliary1}-\eqref{eq:auxiliary2} with $\bm{j}=
	\varepsilon \bm{v}$. Thus, by using the ellipticity of
	$A_{h}$ on $\mr{Ker}(B_h)$ in Lemma \ref{lem:ellipticity}
	and \eqref{eq:auxiliary1}, we obtain
	\begin{align*}
		A_{h}(\bm{u}_{h}  ,\lambda_{h};\bm{v},\eta) - k^{2}(\bm{u}_{h}, 
		\bm{v})_{\varepsilon}
		&=A_{h}(\bm{v},\eta;\bm{v},\eta)
			-k^{2}(\bm{v},\bm{v})_{\varepsilon}  
			+k^{2}\left( A_{h}(\tilde{\bm{u}}_{h},\tilde{\lambda}_{h};
			\bm{v},\eta)
			- k^{2}(\tilde{\bm{u}}_{h}, \bm{v})_{\varepsilon} \right)\\
		&= A_{h}(\bm{v},\eta;\bm{v},\eta) 
		\geq \kappa_{A}'\|(\bm{v},\eta)\|_{W(h)}^{2}.
	\end{align*}
	From Lemma \ref{lem:same solution}, we know
	$\tilde{\lambda}_{h} = \llbracket \tilde{p}_{h} \rrbracket_{N}$. Hence, by
	using the stability estimates
	\eqref{eq:uh-bounded}-\eqref{eq:ph-bounded}, we have
	\begin{equation*}
	\begin{split}
		\|(\tilde{\bm{u}}_{h},\tilde{\lambda}_{h})\|_{W(h)} 
		&=
		\|\tilde{\bm{u}}_{h}\|_{V(h)} + \|\tilde{\lambda}_{h} 
		\|_{M_{h}} \\
		&\leq \|\tilde{\bm{u}}_{h}\|_{V(h)} +\|\tilde{p}_{h}\|_{Q(h)} \\
		& \leq C \|\varepsilon^{\frac{1}{2}}\bm{v}\|_{0,\Omega}
		\leq C\|(\bm{v},\eta)\|_{W(h)}.
	\end{split}
	\end{equation*}
	Now, we conclude from the above two inequalities
	\[
		\inf_{0\neq (\bm{v},\eta)\in \mr{Ker}(B_h)}
		\sup_{0\neq (\bm{u}_{h},\lambda_{h})\in \mr{Ker}(B_h)}
		\frac{A_{h}(\bm{u}_{h},\lambda_{h};\bm{v},\eta) 
		-k^{2}(\bm{u}_{h},\bm{v})_{\varepsilon}}
		{\|(\bm{u}_{h},\lambda_{h})\|_{W(h)}
		\|(\bm{v},\eta)\|_{W(h)}}\geq \kappa_{A},
	\]
	which is equivalent to \eqref{eq:inf-sup-1} since
	$A_{h}(\cdot ,\cdot )$ is symmetric.
\end{proof}

Now, we are ready to prove Theorem \ref{thm:defEnergyNorm}.

\begin{proof} [Proof of Theorem \ref{thm:defEnergyNorm}] 
	Since $k^2$ is not a Maxwell eigenvalue, the inf-sup
	condition \eqref{eq:inf-sup-1} holds true for all $h$ small
	enough. Together with the inf-sup condition of $B_h$
	\eqref{eq:inf-sup}, one can get the same abstract error
	estimates stated in Theorem \ref{thm:defAbstractError} for
	$k=0$ also for the indefinite time-harmonic Maxwell equations
	\eqref{eq:Max} with $k^2$ not a Maxwell eigenvalue. Thus,
	the a priori error bound follows directly from the
	polynomial approximation results \eqref{eq:bestAppr
	1}-\eqref{eq:bestAppr 2} and residual estimates in Corollary
	\ref{cor:res est}.
\end{proof}

\section*{Acknowledgements}

The research of K. Liu is funded by a fellowship from China
Scholarship Council (No. 201806180078), which is gratefully
acknowledged.

\bibliographystyle{IMANUM-BIB}
\bibliography{IMANUM-refs}


\appendix

\section*{Appendix A.\ Smoothed interpolation approximation} 
\label{app:smoothed}

In this section, we shall prove the approximation result stated
in Proposition \ref{pop:Ih}. The basic idea is to use some key
local properties of the mollification $\mc{K}_{\delta}$ without
boundary condition constraints \cite[see][Section 3]{Ern2016},
and the properties of the canonical finite element interpolation
operators.

\subsection*{A.1.\ Local properties of the mollification}

In this section, we restate some local properties of the
mollification, which can be deduced easily from the global
version \cite[see][]{Ern2016}.

To make things more clear, we assume $\delta\in (0,1)$. We only
present the local properties for $\mc{K}_{\delta}^{\mr{g}}$
since the others $ \mc{K}_{\delta}^{\mr{c}},
\mc{K}_{\delta}^{\mr{d}}, \mc{K}_{\delta}^{\mr{b}} $ can be
proven similarly. First, we introduce \cite[(3.4a)]{Ern2016}
\begin{equation*}
	(\mathcal{K}_{\delta}^{\mr{g}}f)(\bm{x}) 
	=\int_{B(\bm{0},1)} \rho(\bm{y}) 
	f(\varphi_{\delta}(\bm{x})+\delta r\bm{y}) \,\mathrm{d}\bm{y}.
\end{equation*}
Here, $\varphi_{\delta}$ is the shrinking mapping introduced in
\cite[(2.1)]{Ern2016}, and $r>0$ such that
$\varphi_{\delta}(\Omega)+B(0,\delta r)\subset \Omega$ for all
$\delta\in (0,1)$. We denote by $\mb{J}_{\delta}(\bm{x})$ the Jacobian of $\varphi_{\delta}(\bm{x})$, which is known to converge uniformly to the identity.

\begin{proposition} [Local boundedness]
	\label{pop:mol-stab} 
	For any $K\in \mathcal{T}_h$, there is a $c>0$ such that 
	\begin{equation}
		\label{eq:mol-stab}
		\|\mathcal{K}_\delta f\|_{s,K}\leq c \|f\|_{s,D_K},
	\end{equation}
	for all $f\in H^s(D_K)$, $\delta\in
	(0,\delta_0]$, and $s\in [0,1]$. Here, $D_K =
	\mathrm{int} \left(\cup_{\bar{K}^{\prime}\cap \bar{K}\neq
	\emptyset}\bar{K}^{\prime} \right)$.
\end{proposition}
\begin{proof}
	The result follows from similar arguments as in the proof of
	the global version \cite[Theorem 3.3]{Ern2016}.
\end{proof}

\begin{proposition}[Local convergence]
	\label{pop:mol-convergence} 
	For any $K\in \mathcal{T}_h$, there is a $c>0$ such that 
	\begin{equation}
		\label{eq:mol-conv}
		\|\mathcal{K}_\delta f-f\|_{s,K} 
		\leq c \delta^{t-s} \|f\|_{t,D_K},
	\end{equation}
	for all $f\in H^t(D_K)$, $\delta\in
	(0,\delta_0]$, and all $0\leq s\leq t\leq 1$.
\end{proposition}
\begin{proof}
	The result follows from similar arguments as in the proof of
	the global version \cite[Theorem 3.5]{Ern2016}.
\end{proof}

The following inverse estimate will be useful, we give a proof
for complements.

\begin{proposition}[Local inverse estimate]
	\label{pop:mol-inverse} 
	For any $K\in \mathcal{T}_h$, there is a $c>0$ such that 
	\begin{equation}
		\label{eq:mol-inverse}
		\|\mathcal{K}_\delta f\|_{s,K}\leq c \delta^{t-s} \|f\|_{t,D_K},
	\end{equation}
	for all $f\in H^s(D_K)$, $\delta\in
	(0,\delta_0]$, and $0\leq t\leq s\leq 1$.
\end{proposition}
\begin{proof}
	We only show 
	\[
		\|\nabla(\mathcal{K}_{\delta}f)\|_{0,K}
		\leq c\delta ^{-1} \|f\|_{0,D_K},
	\]
	since the proposition is a direct result of local stability
	given in Proposition \ref{pop:mol-stab} and the Lions-Peetre
	theorem \cite[Chapter 26]{Tartar2007}. For any $f\in
	H^1(D_K)$, integration by parts with respect to $\bm{y}$
	leads to
	\begin{equation*} 
	\begin{split}
		\nabla(\mathcal{K}_\delta f)(\bm{x}) 
		&= \int_{B(0,1)}\rho(\bm{y}) \nabla (f
		(\varphi_{\delta}(\bm{x})+\delta r \bm{y}))\,\mathrm{d}\bm{y} \\
		&= \int_{B(0,1)}\rho(\bm{y})\mb{J}_\delta^T (\bm{x}) (\nabla f)
		(\varphi_{\delta}(\bm{x})+\delta r \bm{y})\,\mathrm{d}\bm{y} \\
		&= -\mb{J}_\delta^T (\bm{x})(\delta r)^{-1} \int_{B(0,1)}(\nabla\rho)(\bm{y})
		f(\varphi_{\delta}(\bm{x})+\delta r \bm{y})\,\mathrm{d}\bm{y} .
	\end{split}
	\end{equation*}
	By noticing the fact that $\mb{J}_\delta (\bm{x})$ converges
	to the identity uniformly and using Cauchy-Schwarz
	inequality, we have
	\begin{equation*} 
	\begin{split}
		\|\nabla(\mathcal{K}_{\delta}f)\|_{0,K}^2 
		&= \int_{K}\left| \mb{J}_{\delta}^T(\bm{x}) (\delta r)^{-1} 
			\int_{B(0,1)}(\nabla \rho)(\bm{y}) f(\varphi_{\delta}(\bm{x}) + 
			\delta r \bm{y})\,\mathrm{d}\bm{y} \right|^2 \,\mathrm{d}\bm{x} \\
		&\leq |\mb{J}_{\delta}^T|^2(\delta r)^{-2} \int_{K}\left| 
			\int_{B(0,1)}(\nabla \rho)(\bm{y}) f(\varphi_{\delta}(\bm{x}) + 
			\delta r \bm{y})\,\mathrm{d}\bm{y} \right|^2 \,\mathrm{d}\bm{x} \\
		&\leq c\delta ^{-2} 
			\int_{B(0,1)}|(\nabla \rho)(\bm{y})|^2 \,\mathrm{d}\bm{y}
			\int_{K}\int_{B(0,1)} \left| f(\varphi_{\delta}(\bm{x}) + 
			\delta r \bm{y})\right|^2 \,\mathrm{d}\bm{y}  \,\mathrm{d}\bm{x} .
	\end{split}
	\end{equation*}
	We again apply the change of variables $K\ni
	\bm{x}\longmapsto \bm{z} = \varphi_{\delta}(\bm{x})+\delta r
	\bm{y}\in D_K$, thus we have 
	\begin{equation*} 
	\begin{split}
		\|\nabla(\mathcal{K}_{\delta}f)\|_{0,K}^2 
		&\leq c\delta ^{-2} \int_{B(0,1)}\int_{D_K}
			|f(\bm{z})|^2 |J^{-1}_\delta(\bm{z})| \,\mathrm{d}\bm{z} 
			\,\mathrm{d}\bm{y}\\
		&\leq c\delta ^{-2} \|f\|_{0,D_K}^2 .
	\end{split}
	\end{equation*}
\end{proof}

\subsection*{A.2.\ Proof of Proposition \ref{pop:Ih}}

We only show the first approximation property because the others
can be proven similarly. We claim for all $\bm{u}\in
H(\mathrm{curl},\Omega)\cap H^s(\mc{T}_h),\, s\in
[0,\frac{1}{2})$
\begin{equation}
	\label{eq:Ih curl}
	\|\bm{u}-\widetilde{\mathcal{I}}_h^{\mathrm{c}}\bm{u}\|_{0,K} 
		\leq c \left( h_K^{s}\|\bm{v}\|_{s,D_K} 
		+h_K\|\nabla \times \bm{u}\|_{0,D_K} \right).
\end{equation}
To this end, by using the triangle inequality, we have
\begin{equation*}
	\|\bm{u}-\widetilde{\mathcal{I}}_h^c \bm{u}\|_{0,K}
	\leq \|\bm{u}-\mathcal{K}_{\delta}^c\bm{u}\|_{0,K}+
		\|\mathcal{K}_{\delta}^c\bm{u}-\mathcal{I}_h^c
		\mathcal{K}_{\delta}^c\bm{u}\|_{0,K}.
\end{equation*}
From the definition of $\delta(\bm{x})$ \eqref{eq:delta def}, there exist constants $c^{\prime},c^{\prime\prime}$ such that 
\begin{equation}
\label{eq:delta}
	c^{\prime}h_K\leq \delta(\bm{x})\leq c^{\prime\prime}h_K, 
	\quad \forall\, \bm{x}\in D_K.
\end{equation}
By using the local convergence property \eqref{eq:mol-conv} and
\eqref{eq:delta}, it holds that
\begin{equation*}
	\|\bm{u}-\mathcal{K}_{\delta}^c\bm{u}\|_{0,K}
	\leq c \delta^{s}\|\bm{u}\|_{s,D_K}
	\leq c h_K^{s}\|\bm{u}\|_{s,D_K}.
\end{equation*}
From the proof of \cite[Theorem 5.41]{Monk2003}, we obtain
\begin{equation*}
\begin{split}
	\|\mathcal{K}_{\delta}^c\bm{u}-\mathcal{I}_h^c
	\mathcal{K}_{\delta}^c\bm{u}\|_{0,K}
	&\leq c \left( h_K \|\mathcal{K}_{\delta}^{\mathrm{d}} 
	\bm{u}\|_{1,K}	+ h_K^{2} |\mathcal{K}_{\delta}^{\mathrm{d}} 
	\nabla\times \bm{u}|_{1,K} \right) \\
	&\leq c \left( h_K\delta^{s-1} \|\bm{u}\|_{s,D_K}
	+h_K^{2} \delta^{-1} \|\nabla\times \bm{u}\|_{0,D_K} \right) \\
	&\leq c \left( h_K^{s}  \|\bm{u}\|_{s,D_K}
	+h_K  \|\nabla\times \bm{u}\|_{0,D_K} \right),
\end{split}
\end{equation*}
where we have used the local inverse estimate \eqref{eq:mol-inverse} and
\eqref{eq:delta}. 
Combining above two results leads to \eqref{eq:Ih curl}.

\section*{Appendix B.\ Discrete Friedrichs inequality}
\label{app:ellipticity}

An alternative way to prove the ellipticity of $A_h$ on the
kernel of $B_h$ (see Lemma \ref{lem:ellipticity}) is to use the
following discrete Friedrichs inequality, which is stated in
\cite[Lemma 7.6]{Buffa2006}.

\begin{lemma}[Discrete Friedrichs inequality]
	\label{lem:discrete Friedrichs}
	There holds
	\[
		\|\varepsilon^{\frac{1}{2}}\bm{v}\|_{0,\Omega} 
		\leq C |\bm{v}|_{V(h)}, \quad \forall \bm{v} 
		\mbox{ such that } (\bm{v},\eta)\in \mr{Ker}(B_h).
	\]
	Here, the positive constant $C$ is independent of the mesh size.
\end{lemma}
\begin{proof}
	Suppose $(\bm{v},\eta)\in \mr{Ker}(B_h)$. Then, there holds 
	\[
		0=B_h(\bm{v},\eta;q)= -\int_{\Omega}\varepsilon \bm{v}\cdot 
		\nabla_h q \,\mathrm{d}\bm{x}, \qquad \forall 
		q\in Q_h^c = Q_h\cap Q,
	\]
	that is $\bm{v}$ belongs to the following set 
	\begin{equation}
		\label{eq:Kh}
		K_h^{\perp} := \{\bm{w}\in V_h:\ (\varepsilon \bm{w}, \nabla q)=(\bm{w}, \nabla q)_{V(h)}=0, 
		\quad \forall\,\bm{q}\in Q_h^{\mathrm{c}} \}.
	\end{equation}
	Here, $(\cdot,\cdot)_{V(h)}$ is the natural inner product
	which produces $\|\cdot\|_{V(h)}$. As stated in \cite[Lemma
	7.6]{Buffa2006}, we have
	\[
		\|\varepsilon^{\frac{1}{2}}\bm{v}\|_{0,\Omega} 
		\leq C |\bm{v}|_{V(h)}, \quad \forall \bm{v}\in K_h^{\perp}.
	\]
	Hence, we can directly conclude our result.
\end{proof}

\section*{Appendix C.\ Uniform convergence}
\label{app:uniform conv}

We prove Proposition \ref{pop:uni conv}, i.e., the uniform
convergence of $T_h$ to $T$. To this end, we follow the
procedure of \cite[Section 4.2]{Buffa2006}. 

Define the space 
\[
	Z_h:=\left\{ \bm{v}\in V_h:\ (\bm{v},\eta)\in \mr{Ker}(B_h),
		\ \forall\eta\in M_h\right\}.
\]
It is easy to see that $Z_h$ contains the range of $T_h$. We
define $Q_h^c := Q_h\cap Q$ and introduce an orthogonal decomposition of $V_h$:
\[
	V_h = K_h \oplus K_h^{\perp},
\]
where $K_h=\nabla Q_h^c$ and $K_h^{\perp}$ is defined by
\eqref{eq:Kh}.  It follows from the proof of Lemma \ref{lem:discrete Friedrichs} that $Z_h\subseteq K_h^{\perp}$.

Furthermore, from the definition of the solution operators $T,\
T_p$ and $T_h,\ T_{p,h}$ \eqref{eq:TandTp}-\eqref{eq:ThandTph},
we easily obtain the following lemma.

\begin{lemma}
	\label{lem:}
	For all $\bm{f}_h^0\in K_h$, there hold 
	\[
		T \bm{f}_h^0= T_h \bm{f}_h^0=\bm{0},\quad \mbox{and }\quad
		T_{p} \bm{f}_h^0= T_{p,h} \bm{f}_h^0=\bm{f}_h^0.
	\]
\end{lemma}
\begin{proof}
	Since $\bm{f}_h^0\in K_h$, there exists $q\in Q_h^c$ such
	that $\bm{f}_h^0=\nabla q$. It is easy to check that
	$(\bm{0}, \nabla q)$ solves
	\eqref{eq:conForm1}-\eqref{eq:conForm2} and
	\eqref{eq:MDG1}-\eqref{eq:MDG2} simultaneously with
	$\bm{j}=\varepsilon \nabla q$.
\end{proof}

Thus, the uniform convergence \eqref{eq:Th convergence} of $T_h$
to $T$ follows directly from \cite[Proposition 4.4]{Buffa2006},
combining with the regularity estimates given in Lemma
\ref{lem:regularity} and the error estimate given in Theorem
\ref{thm:defEnergyNorm}.

\end{document}